\documentclass{amsart}
\usepackage[utf8]{inputenc}
\usepackage[english]{babel}
\usepackage{amsmath}
\usepackage{amsfonts}
\usepackage{amssymb}
\usepackage{graphicx}
\usepackage[usenames, dvipsnames]{color}
\usepackage[left=1.3cm,right=1.3cm,top=3cm,bottom=3cm]{geometry}
\usepackage{stackrel}
\usepackage{float}

\usepackage{hyperref}

\usepackage[all,cmtip]{xy}
\usepackage{extarrows}

\usepackage{longtable}

\usepackage{amsthm} 
\usepackage{mathtools}
\usepackage{comment}
\usepackage{multicol}

\usepackage{tikz}
\newcommand*\circled[1]{\tikz[baseline=(char.base)]{
            \node[shape=circle,draw,inner sep=1pt] (char) {#1};}}

\counterwithin{equation}{section}

\newtheorem{theorem}{Theorem}[section]
\newtheorem{definition}[theorem]{Definition}
\newtheorem{proposition}[theorem]{Proposition}
\newtheorem{lemma}[theorem]{Lemma}
\newtheorem{corollary}[theorem]{Corollary}

\theoremstyle{remark}
\newtheorem{remark}[theorem]{Remark}
\newtheorem{example}[theorem]{Example}
\newtheorem{notation}[theorem]{Notation}

\newcommand{\AAA}{\mathbb{A}}
\newcommand{\BB}{\mathbb{B}}
\newcommand{\CC}{\mathbb{C}}
\newcommand{\PP}{\mathbb{P}}

\newcommand{\QQ}{\mathbb{Q}}
\newcommand{\ZZ}{\mathbb{Z}}

\newcommand{\B}{\mathrm{B}}

\newcommand{\cA}{\mathcal{A}}
\newcommand{\cB}{\mathcal{B}}
\newcommand{\cC}{\mathcal{C}}
\newcommand{\cD}{\mathcal{D}}

\newcommand{\cF}{\mathcal{F}}
\newcommand{\cG}{\mathcal{G}}
\newcommand{\cL}{\mathcal{L}}
\newcommand{\cO}{\mathcal{O}}
\newcommand{\cQ}{\mathcal{Q}}
\newcommand{\cX}{\mathcal{X}}
\newcommand{\cU}{\mathcal{U}}
\newcommand{\cZ}{\mathcal{Z}}

\newcommand{\cfa}{\mathfrak{a}}
\newcommand{\cfb}{\mathfrak{b}}
\newcommand{\cfc}{\mathfrak{c}}
\newcommand{\cfh}{\mathfrak{h}}

\newcommand{\rat}{\mathrm{rat}}
\newcommand{\alg}{\mathrm{alg}}
\newcommand{\num}{\mathrm{num}}
\newcommand{\nil}{\mathrm{nil}}

\DeclareMathOperator{\nc}{nc}

\newcommand{\dg}{\mathrm{dg}}

\DeclareMathOperator{\obj}{obj}
\DeclareMathOperator{\Hom}{Hom}

\DeclareMathOperator{\colim}{colim}
\DeclareMathOperator{\dgcat}{dgcat}
\DeclareMathOperator{\perf}{perf}

\DeclareMathOperator{\perfdg}{perf_{\dg}}
\DeclareMathOperator{\Spec}{Spec}
\DeclareMathOperator{\NChow}{NChow}
\DeclareMathOperator{\Chow}{Chow}
\DeclareMathOperator{\ch}{ch}
\DeclareMathOperator{\td}{td}
\DeclareMathOperator{\Gr}{Gr}
\DeclareMathOperator{\Pf}{Pf}
\DeclareMathOperator{\Vect}{Vect}

\allowdisplaybreaks

\begin{document}
\title[A refinement of some previous results of Bernardara-Marcolli-Tabuada and Ornaghi-Pertusi]{A refinement of some previous results of \\Bernardara-Marcolli-Tabuada and Ornaghi-Pertusi}
\author{José Francisco Reis}
\address{Center for Mathematics and Applications (NovaMath), FCT NOVA and Department of Mathematics, FCT NOVA}
\email{jfd.reis@campus.fct.unl.pt}
\thanks{}
\maketitle
\begin{abstract}
The Voevodsky nilpotence conjecture was proved by Bernardara-Marcolli-Tabuada for certain quadric fibrations, intersections of quadrics, linear sections of Grassmannians, linear sections of determinantal varieties, and Moishezon manifolds, and later by Ornaghi-Pertusi for certain cubic fourfolds and Gushel-Mukai fourfolds. In this paper we refine these results by using the algebraic equivalence relation instead of the nilpotence equivalence relation. Along the way, we address also certain cases of K{\"u}chle fourfolds, families of Sextic del Pezzo surfaces and families of Fano fourfolds of K3 type.
\end{abstract}
\pagenumbering{arabic}

\,

\section{Introduction}

Let $k$ be a base field. Given a smooth proper $k$-scheme $X$, let us denote by $\cZ^*(X)_\QQ$ the (graded) $\QQ$-vector space of algebraic cycles on $X$. It is well-known that one can define many different equivalence relations on $\cZ^*(X)_\QQ$, such as the rational equivalence relation $\sim_\rat$, the algebraic equivalence relation $\sim_\alg$, the nilpotence equivalence relation $\sim_\nil$, the numerical equivalence relation $\sim_\num$, etc; consult \cite{Fulton2nd_ed}.
Voevodsky conjectured in \cite{Voevodsky95} that the nilpotence and numerical equivalence relations agree. This conjecture was proved\footnote{Thanks to the work of Kahn-Sebastian, Matsusaka, Voevodsky, and Voisin (see \cite{KS09,Mat57,Voevodsky95, Voisin96}), Voevodsky's conjecture is known in the case of curves, surfaces, and abelian 3-folds (when $k$ is of characteristic zero).} by Bernardara-Marcolli-Tabuada \cite{BMT} for certain quadric fibrations, intersections of quadrics, linear sections of Grassmannians, linear sections of determinantal varieties, and Moishezon manifolds, and later by Ornaghi-Pertusi \cite{OP} for certain cubic fourfolds and Gushel-Mukai fourfolds.
As shown by Voevodsky in \cite{Voevodsky95}, every algebraic cycle which is algebraically trivial is also nilpotently trivial. Consequently, it is natural to ask if the aforementioned results also hold when the nilpotence equivalence relation is replaced by the algebraic equivalence relation? Our main result is an affirmative answer to this question:

\begin{theorem}[Refinement]
\label{theorem: main theorem}
The algebraic and numerical equivalence relations agree for certain quadric fibrations, intersections of quadrics, linear sections of Grassmannians, linear sections of determinantal varieties, Moishezon manifolds, cubic fourfolds and Gushel-Mukai fourfolds. In addition, they also agree for certain K{\"u}chle fourfolds, families of Sextic del Pezzo surfaces and families of Fano fourfolds of K3 type.\label{main theo}
\end{theorem}

Theorem \ref{theorem: main theorem} follows from the combination of Corollaries \ref{corollary: certain quadrics}, \ref{corollary: certain intersections}, \ref{corollary: sextic del Pezzo}, \ref{corollary: grassmannians}, \ref{corollary: determinantal} with Theorems \ref{theorem: certain moishezon}, \ref{theorem (A)}, \ref{theorem (C)}, \ref{theorem (D)}, \ref{theorem: certain kuchel fourfolds} and \ref{theorem: fano fourfolds of K3 type}; consult \S \ref{subsection: severi-brauer}-\S\ref{subsection: HPD} and \S\ref{subsection: fano K3 type} below.

\begin{remark}
\label{remark: not always true}
Theorem \ref{main theo} does not hold for every smooth proper $k$-scheme! For example, it follows from the pioneering work of Griffiths \cite{Grif69} that if $X$ is a general quintic $3$-fold, then the algebraic and numerical equivalence relations on $\cZ^*(X)_\QQ$ do not agree; consult also the later works of Clements \cite{Clem83} and Voisin \cite{Voisin00} for further examples.
\end{remark}

In addition to Theorem \ref{main theo}, we also prove that the quotient between any two equivalence relations is invariant under Homological Projective Duality in the sense of Kuznetsov \cite{Kuzn07}; consult Theorem \ref{Theorem: HPD} below. This allows us to explicitly compute the quotient between the rational and the algebraic equivalence relations for certain 5-folds; consult Corollaries \ref{corollary: Grassmannian curve genus 1}-\ref{corollary: Grassmannian curve genus 43} below.

\section{Preliminaries}

Throughout this paper, $k$ denotes a base field.

\subsection{Dg categories}
\label{subsection: dg}
We recall here some basic concepts on differential graded categories; for a survey, consult \cite{Keller06}. 

Let $(C(k),\otimes, k )$ be the category of (cochain) complexes of $k$-vector spaces. A \textit{differential graded (dg) category} $\cA$ is a category enriched over $C(k)$ and a dg functor $F: \cA \rightarrow \cB$ is a functor enriched over $C(k)$. Note that every $k$-algebra $A$ gives rise to a dg category with a single object.
Recall also that, given any quasi-compact quasi-separated $k$-scheme $X$, the category of perfect complexes $\perf (X)$ admits a canonical\footnote{When $X$ is quasi-projective, this dg enhancement is moreover unique, see \cite[Theorem 2.12]{LunOrl}} dg enhancement $\perfdg (X)$; see \cite[\S 4.6]{Keller06}.

Let $\cA$ be a dg category. The opposite dg category $\cA^{\mathrm{op}}$ has the same objects of $\cA$ and $\cA^{\mathrm{op}}(x,y):=\cA(y,x)$ and the category $\mathrm{H}^0(\cA)$ has the same objects of $\cA$ and $(\mathrm{H}^0(A))(x,y)=\mathrm{H}^0(\cA(x,y))$; consult \cite[\S 2.2]{Keller06}.
A \textit{right dg $\cA$-module} is a dg functor $M: \cA^{\mathrm{op}} \rightarrow \cC_{\dg}(k)$ with values in the dg category of complexes of $k$-vector spaces.
The category of dg modules $\cC(\cA)$ has the right dg $\cA$-modules as objects and the morphisms of dg functors as morphisms. The localization of $\cC(\cA)$ with respect to the class of quasi-isomorphism is called the \emph{derived category} $\cD(\cA)$ of $\cA$. We will write $\cD_c (\cA)$ for the full subcategory of compact objects of $\cD(\cA)$; consult \cite[\S 3.2]{Keller06}.

A dg functor $F: \cA \rightarrow \cB$ is a \textit{Morita equivalence} if it induces an equivalence on the derived categories $\cD(\cA) \simeq \cD(\cB)$; see \cite[\S 4.6]{Keller06}.
Given dg categories $\cA$ and $\cB$, its \textit{tensor product} $\cA \otimes \cB$ is the dg category whose set of objects is $\obj(\cA) \times \obj(\cB)$ and $(\cA \otimes \cB)((x,w),(y,z)):= \cA(x,y) \otimes \cB(w,z)$. Following \cite[\S 2.3]{Keller06}, this construction gives a symmetric monoidal structure on the category $\dgcat(k)$ of (essentially small) dg categories and dg functors over the base field $k$.
A \emph{dg $\cA$-$\cB$-bimodule} is a dg functor $\B: \cA \otimes \cB^{\mathrm{op}} \rightarrow \cC_{\dg}(k)$, i.e., a right dg $\cA^{\mathrm{op}} \otimes \cB$-module. Given a dg functor $F: \cA \rightarrow \cB$ there exists an induced dg $\cA$-$\cB$-bimodule associated to $F$ defined as $\prescript{}{F}{\cB} : \cA \otimes \cB^{\mathrm{op}} \longrightarrow \cC_{\dg}(k),$ $(x,z) \longmapsto \cB(z,F(x))$.

Following Kontsevich \cite{kont05Video, Kont09Chp,kont10Notes}, a dg category $\cA$ is called \textit{smooth} if the dg $\cA$-$\cA$-bimodule $\prescript{}{id}{\cA}$ belongs to the subcategory $\cD_c(\cA^{\mathrm{op}} \otimes \cA)$ and it is called \textit{proper} if $\sum_i \dim \mathrm{H}^i\cA(x,y)<\infty$ for every pair of objects $x,y \in \cA$.
The dg category of the perfect complexes $\perfdg (X)$ associated to a smooth proper $k$-scheme $X$ is an example of a smooth proper dg category; see \cite[Example 1.42(ii)]{Tab15a}. We shall write $\dgcat_{\mathrm{sp}}(k) \subseteq \dgcat(k)$ for the full subcategory of smooth proper dg categories.

Given dg categories $\cA$ and $\cB$ and a $\cB$-$\cA$-bimolude $\B$, consider the dg category $T(\cA, \cB; \B)$ whose set of objects is the disjoint union of the sets of objects of $\cA$ and $\cB$, and whose complexes of morphisms are defined as follows:
\begin{align*}
T(\cA, \cB; \B) (x,y):=\left\lbrace\begin{matrix}
    \cA(x,y) & \text{if} & x,y \in \cA\\
    \cB(x,y) & \text{if} & x,y \in \cB\\
    \B(y,x) & \text{if} & x \in \cA, y \in \cB\\
    0 & \text{if} & y \in \cA, x \in \cB.\\
    \end{matrix}\right. 
\end{align*}

By consctruction, we have two canonical inclusions $i_\cA : \cA \rightarrow T(\cA, \cB; \B)$ and $i_\cB : \cB \rightarrow T(\cA, \cB; \B)$.
\begin{lemma}
\label{Lemma Triangular-system}
Let $\cA$ and $\cB$ be two dg categories and $\B$ a dg $\cB$-$\cA$-bimolude. Given a dg category $\cC$ we have a canonical identification between the dg categories $T(\cA, \cB; \B) \otimes \cC$ and $T(\cA \otimes \cC, \cB \otimes \cC; \B \otimes \cC)$.
\end{lemma}

\begin{proof}
The proof is simple and we leave it to the reader.
\end{proof}

\subsection{Noncommutative motives}
\label{subsection: nc motives}
We recall here some basic concepts on noncommutative motives. For a book, resp. survey, consult \cite{Tab15a}, resp. \cite{Tab18a}. Recall from \cite[\S 4.1]{Tab15a} the construction of the category of \textit{noncommutative Chow motives} with $\QQ$-coefficients $\NChow(k)_\QQ$. 
This category is $\QQ$-linear, additive, idempotent complete, rigid symmetric monoidal (i.e., all its objects are dualizable) and is equipped with a symmetric monoidal functor $U(-)_\QQ: \dgcat_{\mathrm{sp}}(k) \rightarrow \NChow(k)_\QQ$. Moreover, given any two smooth proper dg categories $\cA$ and $\cB$, we have the following computation:
\begin{align}
\label{Hom Nchow}
    \Hom_{\NChow(k)_\QQ}(U(\cA)_\QQ, U(\cB)_\QQ):=K_0(\cD_c(\cA^{\mathrm{op}} \otimes \cB))_\QQ=K_0(\cA^{\mathrm{op}} \otimes \cB)_\QQ .
\end{align}
Furthermore, the composition law in $\NChow(k)_\QQ$ is induced by the (derived) tensor product of bimodules.

Finally, recall from \cite[\S 4.4 \& \S4.6]{Tab15a} the construction of the categories of noncommutative Voevodsky motives NVoev$(k)_\QQ$ and noncommutative numerical motives NNum$(k)_\QQ$. These categories are also $\QQ$-linear, additive, idempotent complete, and rigid symmetric monoidal.

\section{Equivalence relations}\label{equiv rel}

Let $\cA$ be a smooth proper dg category. In this section, we introduce three different equivalence relations on the rational Grothendieck group $K_0(\cA)_\QQ:=K_0( \cD_c (\cA))_\QQ$ and compare them with their classical commutative counterparts.

\subsection{Algebraic equivalence relation}
\label{subsection alg}
Given a smooth connected $k$-scheme $T$ and a $k$-rational point $p: \Spec (k) \rightarrow T$, consider the associated pull-back dg functor $p^*: \perfdg (T) \rightarrow \perfdg (\Spec (k))$. Note that $\perfdg (\Spec (k))$ and $k$ are Morita equivalent.

\begin{definition}
\label{def: algebrico}
An element $\alpha \in K_0(\cA)_\QQ$ is called \emph{algebraically trivial} if it belongs to the image of the homomorphism
\begin{align}
    \bigoplus_{(T,p,q)} K_0( \cA \otimes \perfdg(T))_\QQ \xrightarrow[]{\bigoplus (K_0(id \otimes p^*)_\QQ - K_0(id \otimes q^*)_\QQ)} K_0(\cA)_\QQ,
    \label{alg rel}
\end{align}
where the direct sum ranges over all smooth connected $k$-schemes $T$ equipped with two $k$-rational points $p$ and $q$.
\end{definition}

\begin{remark}
\label{remark: k alg closed}
In the particular case where $k$ is algebraically closed, any two $k$-rational points of $T$ can be joined by a smooth projective $k$-curve. Hence, in this particular case, it suffices to consider the triples $(C,p,q)$ where $C$ is a smooth projective $k$-curve equipped with two $k$-rational points $p$ and $q$.
\end{remark}

The above homomorphism (\ref{alg rel}) gives rise to the \textit{algebraic equivalence relation} $\sim_\alg$ on $K_0(\cA)_\QQ$. In what follows, we will write $K_0(\cA)_\QQ /_{\!\sim_\alg}$ for the associated quotient, i.e., for the cokernel of (\ref{alg rel}).

\begin{lemma}
\label{lemma: k algebrico sends morita to iso}
The functor $K_0(-)_\QQ/_{\!\sim_\alg} : \dgcat_{\mathrm{sp}}(k) \rightarrow \Vect (\QQ)$ sends Morita equivalences to isomorphisms.
\end{lemma}
\begin{proof}

Let $F: \cA \rightarrow \cB$ be a Morita equivalence in $\dgcat_{\mathrm{sp}}(k)$. Since the functor $K_0(-)_\QQ$ is an additive invariant (consult \cite[\S 2]{Tab15a}), $K_0(F)_\QQ$ is an isomorphism. 
Moreover, since Morita equivalences are preserved by tensor products over a field (consult \cite[\S 1.6.4]{Tab15a}), we have a Morita equivalence $F \otimes id_{\perfdg(T)}$ and, consequently, an isomorphism $K_0(F \otimes id_{\perfdg(T)})_\QQ$. We then obtain the following commutative diagram:

\begin{align}
\label{diagram: k alg sends morita to iso}
\begin{split}
\xymatrix@C=5pc @R=4pc{
\bigoplus_{(T,p,q)}K_0(\cA \otimes \perfdg(T))_\QQ  \ar[d]_-{\bigoplus K_0(F \otimes id_{\perfdg(T)})_\QQ}^-{\simeq} \ar[r]^-{(\ref{alg rel})} & K_0(\cA)_\QQ \ar[d]^-{K_0(F)_\QQ}_-{\simeq} \ar@{->>}[r]  & K_0(\cA)_\QQ/_{\!\sim_\alg}                               \ar[d]^{K_0(F)_\QQ/_{\!\sim_\alg}}  \\
\bigoplus_{(T,p,q)}K_0(\cB \otimes \perfdg(T))_\QQ  \ar[r]_-{(\ref{alg rel})}  & K_0(\cB)_\QQ \ar@{->>}[r]  & K_0(\cB)_\QQ/_{\!\sim_\alg}.}
\end{split}
\end{align}

Thanks to Definition \ref{def: algebrico}, we hence conclude from (\ref{diagram: k alg sends morita to iso}) that $K_0(F)_\QQ/_{\!\sim_\alg}$ is an isomorphism.\end{proof}

\begin{lemma}
\label{lemma: K algebrico 2nd prop additive}
Let $\cA$ be a smooth proper dg category such that $\mathrm{H}^0(\cA)=\langle \cfb, \cfc \rangle$ admits a semi-orthogonal decomposition in the sense of Bondal-Orlov \cite{BondalOrlov02}. Let $\cfb^{\dg}$ and $\cfc^{\dg}$ denote, respectively, the dg enhancement of $\cfb$ and $\cfc$ induced from $\cA$.
Under these assumptions, the inclusions of $\cfb^{\dg}$ and $\cfc^{\dg}$ into $\cA$ induce an isomorphism:
\begin{align*}
    K_0(\cfb^{\dg})_\QQ/_{\!\sim_\alg} \oplus  K_0(\cfc^{\dg})_\QQ/_{\!\sim_\alg} \stackrel{\simeq}{\longrightarrow} K_0(\cA)_\QQ/_{\!\sim_\alg}.
\end{align*}
\end{lemma} 
\begin{proof}
Consider the dg category $T(\cfb^{\dg},\cfc^{\dg}; \prescript{}{id}{\cA})$, which is known to be smooth and proper; see \cite[Proposition 4.9]{KuznLunts15}. 
Following the proof of \cite[Proposition 2.2]{Tab15a}, the inclusion of $T(\cfb^{\dg},\cfc^{\dg}; \prescript{}{id}{\cA})$ into $\cA$ is a Morita equivalence.
Therefore, since the functor $K_0(-)_\QQ$ is an additive invariant, we obtain an induced isomorphism between $K_0(T(\cfb^{\dg},\cfc^{\dg};\prescript{}{id}{\cA}))_\QQ$ and $K_0(\cA)_\QQ$.
Consequently, in order to finish the proof, it suffices to show that the inclusions of $\cfb^{\dg}$ and $\cfc^{\dg}$ into $T(\cfb^{\dg},\cfc^{\dg}; \prescript{}{id}{\cA})$ induce an isomorphism 
\begin{align}
\label{useful iso}
K_0(\cfb^{\dg})_\QQ/_{\!\sim_\alg} \oplus  K_0(\cfc^{\dg})_\QQ/_{\!\sim_\alg} \longrightarrow K_0(T(\cfb^{\dg},\cfc^{\dg}; \prescript{}{id}{\cA}))_\QQ/_{\!\sim_\alg}.  
\end{align}

Consider the following commutative diagram
\begin{align*}
\resizebox{19.024cm}{!}{
\xymatrix@C=2.5pc @R=4pc{
 (\bigoplus_{(T,p,q)} K_0(\cfb^{\dg} \otimes \perfdg T)_\QQ) \oplus (\bigoplus_{(T,p,q)} K_0(\cfc^{\dg} \otimes \perfdg T)_\QQ) \ar[r]^-{(\ref{alg rel}) \oplus (\ref{alg rel})} \ar[d]  & K_0(\cfb^{\dg})_\QQ \oplus K_0(\cfc^{\dg})_\QQ \ar@{->>}[r] \ar[d]_-\simeq & K_0(\cfb^{\dg})_\QQ/_{\!\sim_\alg} \oplus K_0(\cfc^{\dg})_\QQ/_{\!\sim_\alg} \ar[d]^{(\ref{useful iso})}\\
\bigoplus_{(T,p,q)} K_0(T(\cfb^{\dg},\cfc^{\dg}; \prescript{}{id}{\cA}) \otimes \perfdg T)_\QQ \ar[r]_-{(\ref{alg rel})} & K_0(T(\cfb^{\dg},\cfc^{\dg}; \prescript{}{id}{\cA}))_\QQ \ar@{->>}[r] & K_0(T(\cfb^{\dg},\cfc^{\dg}; \prescript{}{id}{\cA}))_\QQ/_{\!\sim_\alg}}}
\end{align*}
\noindent with exact rows and whose vertical arrows are induced by the inclusions of $\cfb^{\dg}$ and $\cfc^{\dg}$ into $T(\cfb^{\dg},\cfc^{\dg}; \prescript{}{id}{\cA})$. We need to prove that (\ref{useful iso}) is an isomorphism.
By diagram chasing, we observe that (\ref{useful iso}) is surjective. Since the middle vertical map of the above diagram is an isomorphism, Lemma \ref{Lemma Triangular-system} implies that, for a fixed $(T,p,q)$, the corresponding map
\begin{align*}
K_0(\cfb^{\dg} \otimes \perfdg T)_\QQ) \oplus K_0(\cfc^{\dg} \otimes \perfdg T)_\QQ) \longrightarrow K_0(T(\cfb^{\dg},\cfc^{\dg}; \prescript{}{id}{\cA}) \otimes \perfdg T)_\QQ    
\end{align*}
is an isomorphism as well.
This implies that the left-hand-side vertical map of the above diagram is surjective. Consequently, we conclude (once again by diagram chasing) that (\ref{useful iso}) is moreover injective and hence an isomorphim.
\end{proof}

\begin{lemma}
\label{lemma: azumaya algebras}
Let $X$ be a smooth proper $k$-scheme and $\BB_0$ an Azumaya algebra over $X$. The canonical dg functor

\noindent $i_{X,\BB_0}: \perfdg(X) \rightarrow \perfdg(X, \BB_0)$ induces an isomorphism:
\begin{align*}
    K_0(i_{X,\BB_0})_\QQ/_{\!\sim_\alg}: K_0(\perfdg(X))_\QQ/_{\!\sim_\alg} \xrightarrow[]{\simeq}K_0(\perfdg(X, \BB_0))_\QQ/_{\!\sim_\alg}.
\end{align*}
\end{lemma}
\begin{proof}

Thanks to \cite[Propositions 8.3 and 8.17]{TabVB15}, the canonical dg functor $i_{X,\BB_0}$ induces an isomorphism $U(i_{X,\BB_0})_\QQ$ in the category of noncommutative Chow motive $\NChow(k)_\QQ$.
As a consequence, since the functor $U(-)_\QQ$ is symmetric monoidal, the morphism $U(i_{X,\BB_0} \otimes id_{\perfdg(T)})_\QQ$ is also an isomorphism.
Thanks to the computation (\ref{Hom Nchow}), this implies that the following two maps are invertible:
\begin{align*}
        K_0(i_{X,\BB_0})_\QQ: K_0(\perfdg(X))_\QQ &\stackrel{\simeq}{\longrightarrow}  K_0(\perfdg(X, \BB_0))_\QQ;\\
    K_0(i_{X,\BB_0} \otimes id_{\perfdg(T)})_\QQ: K_0(\perfdg(X) \otimes \perfdg (T))_\QQ &\stackrel{\simeq}{\longrightarrow}  K_0(\perfdg(X, \BB_0) \otimes \perfdg (T))_\QQ.
\end{align*}

Therefore, we obtain the following commutative diagram:

\begin{align}
\label{diagram: azumaya}
\begin{split}
\xymatrix@C=3pc @R=3.5pc{
\bigoplus_{(T,p,q)}K_0(\perfdg(X) \otimes \perfdg(T))_\QQ  \ar[d]_-{\bigoplus K_0(i_{X,\BB_0} \otimes id_{\perfdg(T)})_\QQ}^{\simeq} \ar[r]^-{(\ref{alg rel})} & K_0(\perfdg(X))_\QQ    \ar[d]^-{K_0(i_{X,\BB_0})_\QQ}_-{\simeq} \ar@{->>}[r]  & K_0(\perfdg(X))_\QQ/_{\!\sim_\alg}                               \ar[d]^{K_0(i_{X,\BB_0})_\QQ/_{\!\sim_\alg}}  \\
\bigoplus_{(T,p,q)}K_0(\perfdg(X, \BB_0) \otimes \perfdg(T))_\QQ  \ar[r]_-{(\ref{alg rel})}  & K_0(\perfdg(X, \BB_0))_\QQ \ar@{->>}[r]  & K_0(\perfdg(X, \BB_0))_\QQ/_{\!\sim_\alg}.}
\end{split}
\end{align}
Thanks to Defintion \ref{def: algebrico}, we hence conclude from (\ref{diagram: azumaya}) that $K_0(i_{X,\BB_0})_\QQ/_{\!\sim_\alg}$ is an isomorphism.
\end{proof}

\subsection{Nilpotence equivalence relation}
\label{subsection nil}
Given an integer $m\geq 1$ consider the following functor:
\begin{align}
\prod _{i=1}^m\cD_c(\cA) \longrightarrow \cD_c(\cA^{\otimes m}) \qquad
\{M_i\}_{1 \leq i \leq m}  \longmapsto \otimes_{i=1}^m M_i.
\label{step to nil rel}
\end{align}
The functor (\ref{step to nil rel}) is triangulated in each one of its variables. Hence, it gives rise to the following multilinear pairing:
\begin{align}
\prod _{i=1}^m K_0(\cA)_\QQ \longrightarrow K_0(\cA^{\otimes m})_\QQ \qquad
\{[M_i]\}_{1 \leq i \leq m}  \longmapsto [\otimes_{i=1}^m M_i].
\label{nil rel}
\end{align}

\begin{definition}
An element $\alpha \in K_0(\cA)_\QQ$ is called \emph{nilpotently trivial} if there is an integer $m \geq 1$ such that the image of $\underbrace{(\alpha, \ldots, \alpha)}_{m\text{-copies}} $ under the multilinear pairing (\ref{nil rel}) is equal to $0$.
\end{definition}

The above multilinear pairing (\ref{nil rel}) gives rise to the \textit{nilpotence equivalence relation} $\sim_\nil$ on $K_0(\cA)_\QQ$. In what follows, we will write $K_0(\cA)_\QQ/_{\!\sim_\nil}$ for the associated quotient.

\begin{remark}
\label{remark: nil nchow}
Recall from \S \ref{subsection: nc motives} the following computation in the category of noncommutative Chow motives:
\begin{align*}
    \Hom_{\NChow(k)_\QQ} ( U(k)_\QQ, U(\cA)_\QQ)=K_{0}(\cD_c(k^{\mathrm{op}} \otimes \cA))_\QQ=K_{0}(\cA)_\QQ.
\end{align*}
This enables the following re-phrasing of the nilpotence equivalence relation: an element $\alpha \in K_0(\cA)_\QQ$ is nilpotently trivial if and only if the associated morphism $\alpha: U(k)_\QQ \rightarrow U(\cA)_\QQ$ in $\NChow (k)_\QQ$ is $\otimes $-nilpotent, or, equivalently, if and only if the associated morphism $\alpha: U(\cA^{\mathrm{op}})_\QQ \rightarrow U(k)_\QQ$ in $\NChow (k)_\QQ$ is $\otimes $-nilpotent.
\end{remark}

\subsection{Numerical equivalence relation}
\label{subsection num}

Consider the following Euler bilinear pairing:
\begin{align}
    \chi: K_0(\cA)_\QQ \times K_0(\cA)_\QQ \longrightarrow \QQ \qquad
    ( [M],[N]) \longmapsto \sum_{n} (-1)^n \dim_k \Hom_{\cD_c(\cA)}(M, N[n]).
\label{euler pairing}
\end{align}

Although the bilinear pairing (\ref{euler pairing}) is not symmetric nor skew-symmetric, its left and right kernels agree; consult \cite[\S 4]{Tab15a}. Let us then write $\ker(\chi)$ for this (unique) kernel.

\begin{definition}
An element $\alpha \in K_0(\cA)_\QQ$ is called \emph{numerically trivial} if it belongs to $\ker(\chi)$. In other words, we have $\chi (\alpha, \beta)=0$ (or, equivalently, $\chi (\beta, \alpha)=0$) for every $\beta \in K_0(\cA)_\QQ$.
\label{num rel}
\end{definition}

Definition \ref{num rel} gives rise to the \textit{numerical equivalence relation} $\sim_\num$ on $K_0(\cA)_\QQ$. In what follows we will write $K_0(\cA)_\QQ/_{\!\sim_\num}$ for the associated  quotient.

\subsection{Comparison between the different equivalence relations}

In this subsection we compare the different equivalence relations introduced in \S \ref{subsection alg}- \S \ref{subsection num}.
\begin{theorem}
\label{ordering}
Given an element $\alpha \in K_0(\cA)_\QQ$, we have the implications $\alpha \sim_\alg 0 \Rightarrow \alpha \sim_\nil 0 \Rightarrow \alpha \sim_\num 0$. 
\end{theorem}

\begin{proof}
We start by proving the implication $\alpha \sim_\alg 0 \Rightarrow \alpha \sim_\nil 0$. Thanks to Proposition \ref{prop: algebraic closure} below, we can assume without loss of generality that $k$ is algebraically closed. Take $\alpha \in K_0(\cA)_\QQ$ an algebraically trivial element.
Recall from Remark \ref{remark: nil nchow} that it is enough to show that the associated morphism $\alpha : U(\cA^{\mathrm{op}})_\QQ \rightarrow U(k)_\QQ$ in the category $\NChow(k)_\QQ$ is $\otimes$-nilpotent.
Since $k$ is algebraically closed and the nilpotently trivial elements form a $\QQ$-linear subspace of $K_0(\cA)_\QQ$, we can assume that there exists a single smooth projective curve $C$, equipped with two $k$-rational points $p$ and $q$, and a element $\beta \in K_0(\cA \otimes \perfdg (C))_\QQ$ such that $\alpha= K_0(id_\cA \otimes p^*)_\QQ(\beta)-K_0(id_\cA \otimes q^*)_\QQ(\beta)$; see Remark \ref{remark: k alg closed}.
Making use of the following computations
\begin{align*}
\Hom_{\NChow(k)_\QQ}(U(\cA^{\mathrm{op}})_\QQ,U(\perfdg (C))_\QQ)\simeq K_0(\cA \otimes_k \perfdg (C))_\QQ \qquad \Hom_{\NChow(k)_\QQ}(U(\cA^{\mathrm{op}})_\QQ,U(k)_\QQ)\simeq K_0(\cA),
\end{align*}
we observe that the associated morphism $\alpha : U(\cA^{\mathrm{op}})_\QQ \rightarrow U(k)_\QQ$ can be written as the following composition
\begin{align*}
    \alpha : U(\cA^{\mathrm{op}})_\QQ \xrightarrow{\quad \beta \quad} U(\perfdg (C))_\QQ  \xrightarrow{U(p^*)_\QQ - U(q^*)_\QQ} U(k)_\QQ,
\end{align*}
where $\beta$ is the morphism associated to the element $\beta \in K_0(\cA \otimes \perfdg (C))_\QQ$.
Consequently, once we show that $U(p^*)_\QQ - U(q^*)_\QQ$ is $\otimes$-nilpotent, we conclude that $\alpha$ is $\otimes$-nilpotent too.
Let $\Chow(k)_\QQ$ be the classical category of Chow motives; see Manin \cite{Manin}. This category is $\QQ$-linear, additive, idempotent complete, and rigid symmetric monoidal.
In addition, it has a symmmetric monoidal functor $\cfh(-)_\QQ: \text{SmProp}(k)^{\mathrm{op}} \rightarrow \Chow(k)_\QQ$ defined on smooth proper $k$-schemes. Following \cite[Theorem 4.3]{Tab15a}, there exists a $\QQ$-linear, fully-faithfull, symmmetric monoidal functor $\Phi$ making the following diagram commutative:

\begin{align}
\label{diagram: chow and nchow}
\begin{split}
    \xymatrix@C=4pc @R=2pc{
\text{SmProp}(k)^{\mathrm{op}} \ar[d]_{\cfh(-)_\QQ} \ar[r]^{X \longmapsto \perfdg (X)} & \dgcat_{\mathrm{sp}}(k) \ar[dd]^{U(-)_\QQ}\\
\Chow(k)_\QQ \ar[d]_{\tau} & \\
\Chow(k)_\QQ/(-\otimes \QQ(1)) \ar[r]_-{\Phi} & \NChow(k)_\QQ,}
\end{split}
\end{align}

\noindent where we denote by $\Chow(k)_\QQ/(-\otimes \QQ(1))$ the orbit category with respect to the Tate motive $\QQ(1)$; consult \cite[\S 4.2]{Tab15a} for the definition of the orbit category. In \cite[Proposition 3.1]{Voevodsky95} it is proven that the morphism $\cfh(p)_\QQ-\cfh(q)_\QQ: \cfh(C)_\QQ \rightarrow \cfh( \Spec (k))_\QQ$, corresponding to the degree zero cycle $p-q$ on $C$, is $\otimes$-nilpotent.
Since both functors $\tau$ and $\Phi$ are symmetric monoidal, the commutativity of diagram (\ref{diagram: chow and nchow}) hence implies that $U(p^*)_\QQ-U(q^*)_\QQ$ is $\otimes$-nilpotent. This concludes the proof.

We now prove the implication $\alpha \sim_\nil 0 \Rightarrow \alpha \sim_\num 0$. Recall first from \cite[\S 5 and Proposition 6.2]{MTab12} that an element $\alpha \in K_0(\cA)_\QQ$ is numerically trivial if and only if for every $\beta \in K_0(\cA^{\mathrm{op}})_\QQ$ the associated composition 
\begin{align*}
    U(k)_\QQ \xrightarrow{\quad \alpha \quad } U(\cA)_\QQ \xrightarrow{\quad \beta \quad } U(k)_\QQ
\end{align*}
is equal to zero.
Let $\alpha\in K_0(\cA)_\QQ$ be a nilpotently trivial element. Bearing in mind Remark \ref{remark: nil nchow}, the associated morphism $\alpha :U(k)_\QQ \rightarrow U(\cA)_\QQ$ is $\otimes$-nilpotent.
Let us assume by absurd that $\alpha$ is not numerically trivial.
In this case, there would exist an element $\beta \in K_0(\cA^{\mathrm{op}})_\QQ$ such that the associated morphism $\beta : U(\cA)_\QQ \rightarrow U(k)_\QQ$ composed with $\alpha$ is different from zero. 
Using the fact that $\beta \circ \alpha \in \Hom_{\NChow(k)_\QQ}(U(k)_\QQ,U(k)_\QQ)\simeq \QQ$, we can further assume without loss of generality that $\beta \circ \alpha$ is the identity. But, this would then imply that $\beta \circ \alpha$ is not $\otimes$-nilpotent, which contradicts the assumption that $\alpha$ is $\otimes$-nilpotent. In conclusion, $\alpha$ is also numerically trivial.
\end{proof}

\begin{proposition}\label{prop: algebraic closure}
Let $\overline{k}/k$ be a fixed algebraic closure of $k$. Given an element $\alpha \in K_0(\cA)_\QQ$, the following holds:
\begin{align*}
\alpha \sim_\alg 0 \Rightarrow (\alpha \otimes_k \overline{k}) \sim_\alg 0 \qquad \text{ and } \qquad \alpha \sim_\nil 0 \Leftrightarrow (\alpha \otimes_k \overline{k}) \sim_\nil 0.
\end{align*}
\end{proposition}
\begin{proof}
It follows from \cite[Proposition 7.2]{MTab14} that if $\cA \in \dgcat_{\mathrm{sp}}(k)$, then $\cA \otimes_k \overline{k} \in \dgcat_{\mathrm{sp}}(\overline{k})$.
Note also that given a smooth connected $k$-scheme $T$, $\perfdg(T) \otimes_k \overline{k}$ is a smooth $\overline{k}$-linear dg category which is canonically Morita equivalent to $\perfdg (\overline{T})$, where $\overline{T}:= T \times_{\Spec (k)} \Spec (\overline{k})$.
Hence, the implication $\alpha \sim_\alg 0 \Rightarrow \alpha \otimes_k \overline{k} \sim_\alg 0$ follows from diagram (\ref{diagram: algebraic closure}) below (where $T'$ is a smooth connected $\overline{k}$-scheme equipped with two $\overline{k}$-rational points $p'$ and $q'$):

\begin{align}
\label{diagram: algebraic closure}
\begin{split}
\xymatrix@C=4pc @R=4pc{\bigoplus_{(T',p',q')} K_0((\cA \otimes_k \overline{k}) \otimes_{\overline{k}} \perfdg (T') )_\QQ \ar[r]^-{(\ref{alg rel})}& K_0(\cA \otimes_k \overline{k})_\QQ\\
\bigoplus_{(T,p,q)} K_0(\cA \otimes_k \perfdg (T) )_\QQ \ar[r]_-{(\ref{alg rel})} \ar[u]^{K_0(-\otimes_k \overline{k})_\QQ} & K_0(\cA)_\QQ \ar[u]_{K_0(-\otimes_k \overline{k})_\QQ}.}
\end{split}
\end{align}

We now prove the equivalence $\alpha \sim_\nil 0 \Leftrightarrow (\alpha \otimes_k \overline{k}) \sim_\nil 0$. For any integer $m\geq 1$, we have the commutative diagram:

\begin{align}
\label{diagram: algebraic closure equivalence nil}
\begin{split}
\xymatrix@C=6.5pc @R=4pc{
\prod_{i=1}^m K_0(\cA \otimes_k \overline{k})_\QQ \ar[r]^-{(\ref{nil rel})}  & K_0((\cA\otimes_k \overline{k})^{\otimes_{\overline{k}} m})_\QQ\\
\prod_{i=1}^m K_0(\cA)_\QQ \ar[r]_-{(\ref{nil rel})} \ar[u]^{\prod_{i=1}^m K_0(-\otimes_k \overline{k})_\QQ} & K_0(\cA^{\otimes_k m})_\QQ \ar[u]_{K_0(-\otimes_k \overline{k})_\QQ}.}
\end{split}
\end{align}

Moreover, we have a canonical identification $\cA^{\otimes_{k}m} \otimes_{k} \overline{k} \simeq (\cA \otimes_{k} \overline{k})^{\otimes_{\overline{k}}m}$. Therefore, by applying Lemma \ref{lemma: alg closure induces injection} below to the dg category $\cA^{\otimes_{k}m}$, we conclude that the homomorphism $K_0(-\otimes_k \overline{k})_\QQ$ is injective. Consequently, the proof follows now from (\ref{diagram: algebraic closure equivalence nil}). 
\end{proof}

\begin{lemma}
\label{lemma: alg closure induces injection}
Given an algebraic closure $\overline{k}/k$, the induced homomorphism $K_0(-\otimes_k \overline{k})_\QQ: K_0(\cA)_\QQ \rightarrow K_0(\cA \otimes_k \overline{k} )_\QQ$ is injective.
\end{lemma}

\begin{proof}
Assume that $\overline{k}/k$ is a finite field extension of degree $d$. In this case, the restriction along the field extension $k \rightarrow \overline{k}$ yields an homomorphism $\text{Res}_{\overline{k}/k}: K_0(\cA \otimes_k \overline{k})_\QQ \rightarrow K_0(\cA)_\QQ$ such that
$\text{Res}_{\overline{k}/k} \circ K_0(- \otimes_k \overline{k})_\QQ$ is equal to multiplication by $d$.
Therefore, since $K_0(\cA)_\QQ$ is a $\QQ$-vector space, the induced homomorphism $K_0(-\otimes_k \overline{k})_\QQ$ is injective.
In the case where the field extension $\overline{k}/k$ is infinite, $\overline{k}$ identifies with the colimit of the filtrant diagram $\{k_i\}_{i \in I}$ of all the intermediate field extensions $\overline{k}/k_i/k$ which are finite over $k$. Since both functors $-\otimes_k \overline{k}$ and $K_0(-)_\QQ$ preserve filtrant (homotopy) colimits, we have that $K_0(\cA \otimes _k \overline{k})_\QQ \simeq \colim_{i \in I} K_0(\cA \otimes_k k_i)_\QQ$ and injectivity of $K_0(-\otimes_k \overline{k})_\QQ$ follows now from the finite field extension case.
\end{proof}

\subsection{Classical equivalence relations on algebraic cycles}

Given a smooth proper $k$-scheme $X$, recall from \cite[Corollary 18.3.2]{Fulton2nd_ed} that the following map is invertible
\begin{align}
\label{iso from nc to c}
    K_0(\perfdg (X) )_\QQ \stackrel{\simeq}{\longrightarrow} \cZ^*(X)_\QQ/_{\!\sim_\rat} \qquad
    [\cF] \longmapsto \ch(\cF) \cdot \sqrt{\td_X},
\end{align}
 where $\ch(\cF)$ stands for the Chern character of $\cF$ and $\td_X$ for the Todd class of $X$.
 
\begin{theorem}
\label{theo: non commutative to classical rel}
The above map (\ref{iso from nc to c}) induces the following isomorphisms:
\begin{align}
\label{iso from nc to c in theorem}
    K_0(\perfdg(X))_\QQ/_{\!\sim_\alg} \xrightarrow{\simeq} \cZ^*(X)_\QQ/_{\!\sim_\alg} \hspace{0.13054cm} K_0(\perfdg(X))_\QQ/_{\!\sim_\nil} \xrightarrow{\simeq} \cZ^*(X)_\QQ/_{\!\sim_\nil} 
    \hspace{0.13054cm} K_0(\perfdg(X))_\QQ/_{\!\sim_\num} \xrightarrow{\simeq} \cZ^*(X)_\QQ/_{\!\sim_\num}.
\end{align}
\end{theorem}

\begin{proof}
The left-hand-side and middle isomorphisms in (\ref{iso from nc to c in theorem}) follow from Lemmas \ref{lemma: chern nc to c}-\ref{lemma: todd nc to c} below.
In order to prove the right-hand-side in (\ref{iso from nc to c in theorem}), note that, given any two perfect complexes $\cF$, $\cG \in \perf(X)$, we have the following equalities
\begin{align}
\label{equation: rephrase euler pairing}
\chi([\cF],[\cG])= K_0(\pi_*([\textbf{R}\underline{\Hom}(\cF,\cG)]))_\QQ=K_0(\pi_*([\cF^* \otimes_X^{\textbf{L}} \cG] ))_\QQ,
\end{align}
where $\pi : X \rightarrow \Spec (k)$ stands for the structure map of $X$ and $\textbf{R}\underline{\Hom}(-,-)$, resp. $(-)^*$, stands for the internal $\Hom$, resp. (contravariant) duality, of the category $\perfdg (X)$. Hence, we obtain the following commmutative diagram:
\begin{align}
\label{diagram: num nc to c}
\begin{split}
    \xymatrix@C=2.5pc @R=1pc{
K_0(\perfdg (X))_\QQ \times K_0(\perfdg (X))_\QQ \ar@{=}[dd] \ar[rrrr]^-{\chi}  & & & & \QQ \ar@{=}[dd] \\
& & \circled{1} & & \\
K_0(\perfdg (X))_\QQ \times K_0(\perfdg (X))_\QQ \ar[rr]_-{K_0((-)^*\otimes^{\textbf{L}}_X -)_\QQ} \ar[dd]_{K_0((-)^*)_\QQ \times id}^{\simeq}& & K_0(\perfdg (X))_\QQ \ar[rr]_-{K_0(\pi_*)_\QQ} \ar@{=}[dd] & & \QQ \ar@{=}[dd]\\
& \circled{2} & & \circled{3} &\\
K_0(\perfdg (X))_\QQ \times K_0(\perfdg (X))_\QQ \ar[rr]^-{-\cdot -} \ar[dd]_{( \ch(-)\cdot \sqrt{\td_X}) \times( \ch(-) \cdot \sqrt{\td_X})}^{\simeq} & & K_0(\perfdg (X))_\QQ \ar[rr]^-{K_0(\pi_*)_\QQ} \ar[dd]_{\simeq}^{ \ch(-)\cdot \td_X}& & \QQ \ar@{=}[dd]\\
& \circled{4} & & \circled{5} &\\
\cZ^*(X)_\QQ/_{\!\sim_\rat} \times \cZ^*(X)_\QQ/_{\!\sim_\rat} \ar[rr]_-{- \cdot -} & & \cZ^*(X)_\QQ/_{\!\sim_\rat} \ar[rr]_-{\pi_{*}} & & \QQ.}
\end{split}
\end{align}

The square $\circled{1}$ is commutative, thanks to (\ref{equation: rephrase euler pairing}). The commutativity of the squares $\circled{2}$ and $\circled{3}$ is obvious; note that $\QQ=K_0(\Spec(k))_\QQ$. 
The square $\circled{4}$ is commutative since $\ch(-)$ is multiplicative.
Finally, the square $\circled{5}$ is commutative because of the Grothendieck-Riemann-Roch formula; see \cite[Theorem 5.26]{Huyb06}. Note that the right vertical map of the square $\circled{5}$, i.e., $\ch(-)\cdot \td_{\Spec (k)}$, is the identity map.
Indeed, since $K_0(\Spec(k))_\QQ=\QQ$ is a graded ring concentrated in degree zero, we have $\td_{\Spec (k)}=1$ and the ring map $\ch(-)$ from $K_0(\Spec(k))_\QQ$ to $\cZ^*(\Spec(k))_\QQ/_{\!\sim_\rat}=\QQ$ is the identity.
We already know that $\ch(-)\cdot \sqrt{\td_X}$ is an isomorphism and that $K_0( (-)^*)_\QQ \times id$ is an isomorphism because both $(-)^*$ and $id$ are isomorphisms. This explains diagram (\ref{diagram: num nc to c}).
Finally, recall that an algebraic cycle $\mu \in \cZ^*(X)_\QQ/_{\!\sim_\rat}$ is called numerically trivial if $\pi_*(\nu \cdot \mu)=0$ for all $\nu \in \cZ^*(X)_\QQ/_{\!\sim_\rat}$. Therefore, given an element $\cF \in K_0(\perfdg(X))_\QQ$, we conclude from the outer commutative square of diagram (\ref{diagram: num nc to c}) that $\cF$ is numerically trivial if $\ch(\cF) \cdot \sqrt{\td_X}$ is numerically trivial. This concludes the proof of Theorem \ref{theo: non commutative to classical rel}.
\end{proof}

\begin{lemma}\label{lemma: chern nc to c}
The Chern character $\ch(-): K_0(\perfdg (X) )_\QQ \stackrel{\simeq}{\longrightarrow} \cZ^*(X)_\QQ/_{ \!\sim_\rat}$ induces the following isomorphisms:
\begin{align}
\label{iso: chern nc to c}
    K_0(\perfdg(X))_\QQ/_{\!\sim_\alg} \stackrel{\simeq}{\longrightarrow} \cZ^*(X)_\QQ/_{\!\sim_\alg} \qquad \qquad K_0(\perfdg(X))_\QQ/_{\!\sim_\nil} \stackrel{\simeq}{\longrightarrow}\cZ^*(X)_\QQ/_{\!\sim_\nil}.
\end{align}
\end{lemma}
\begin{proof}
From \cite[Lemma 4.26]{TabVB18} we have that, for $T$ a smooth connected $k$-scheme, the following dg functor is a Morita equivalence:
\begin{align*}
    - \boxtimes- : \perfdg (X) \otimes \perfdg (T) \longrightarrow \perfdg(X \times T).
\end{align*}
Moreover, by the properties of Chern character (see \cite[page 282, (ii)]{Fulton2nd_ed}), one obtains the commutative diagram:

\begin{align}
\label{diagram: chern I}
\begin{split}
\xymatrix@C=10pc @R=2.5pc{
\bigoplus_{(T,p,q)}K_0(\perfdg (X) \otimes \perfdg(T))_\QQ  \ar[d]_{\bigoplus K_0(-\boxtimes-)_\QQ}^{\simeq} \ar[r]^-{\bigoplus (K_0(id \otimes p^*)_\QQ - K_0(id \otimes q^*)_\QQ)} & K_0(\perfdg(X))_\QQ \ar[dd]_\simeq^{\ch(-)}\\
\bigoplus_{(T,p,q)}K_0(\perfdg(X \times T))_\QQ \ar[d]_{\bigoplus\ch(-)}^{\simeq} & \\
\bigoplus_{(T,p,q)}\cZ^*(X \times T)_\QQ/_{\!\sim_\rat} \ar[r]_-{\bigoplus((id \times p)^*-(id \times q)^*)}    & \cZ^*(X)_\QQ/_{\!\sim_\rat},}
\end{split}
\end{align}
\noindent where $p$ and $q$ stand for $k$-rational points of a smooth connected $k$-scheme $T$. Recall that an algebraic cycle $\mu \in \cZ^*(X)_\QQ/_{\!\sim_\rat}$ is called algebraically trivial if it belongs to the image of $\bigoplus_{(T,p,q)} ((id \otimes p)^* - (id \otimes q)^*)$. Therefore, the left-hand-side isomorphism in (\ref{iso: chern nc to c}) follows from the above commutative diagram (\ref{diagram: chern I}).

In order to prove the right-hand-side isomorphism in (\ref{iso: chern nc to c}), note that a repeated use of \cite[Lemma 4.26]{TabVB18} shows that the following dg functor
\begin{align*}
    - \boxtimes - \cdots - \boxtimes - : \perfdg(X)^{\otimes m} \longrightarrow \perfdg(X^{\otimes m})
\end{align*}
is a Morita equivalence. Therefore, for any integer $m\geq 1$ we obtain the following commutative diagram:

\begin{align}
\label{diagram: chern II}
\begin{split}
\xymatrix@C=11pc @R=2.25pc{
\prod_{i=1}^m K_0(\perfdg(X))_\QQ \ar[r]^-{([\cF_1],\ldots, [\cF_m]) \longmapsto [\cF_1 \otimes \ldots \otimes \cF_m]} \ar[dd]_{\prod \ch(-)}^{\simeq}& K_0(\perfdg (X)^{\otimes m})_\QQ \ar[d]_{\simeq}^{K_0(- \boxtimes - \ldots - \boxtimes -)_\QQ}\\
 & K_0(\perfdg(X^{\times m}))_\QQ \ar[d]_{\simeq}^{\ch(-)}\\
 \prod_{i=1}^m \cZ^*(X)_\QQ/_{\!\sim_\rat} \ar[r]_-{(\mu_1,\ldots, \mu_m) \longmapsto \mu_1 \times \ldots \times \mu_m}  & \cZ^*(X^{\times m})/_{\!\sim_\rat}.
 }
\end{split}
\end{align}

Recall that an algebraic cycle $\mu \in \cZ^*(X)_\QQ/_{\!\sim_\rat}$ is called nilpotently trivial if there is a positive integer $m$ such that $\mu^{m}$ is equal to $0$. Therefore, the right-hand-side isomorphism in (\ref{iso: chern nc to c}) follows from the above commutative diagram (\ref{diagram: chern II}).
\end{proof}

\begin{lemma}
\label{lemma: todd nc to c}
The homomorphism $- \cdot \sqrt{\td_{X}}: \cZ^*(X)_\QQ/_{ \!\sim_\rat} \stackrel{\simeq}{\longrightarrow} \cZ^*(X)_\QQ/_{ \!\sim_\rat}$ induces the following isomorphisms:
\begin{align}
\label{iso: todd nc to c}
    \cZ^*(X)_\QQ/_{\!\sim_\alg} \stackrel{\simeq}{\longrightarrow} \cZ^*(X)_\QQ/_{\!\sim_\alg} \qquad \qquad \cZ^*(X)_\QQ/_{\!\sim_\nil} \stackrel{\simeq}{\longrightarrow} \cZ^*(X)_\QQ/_{\!\sim_\nil}.
\end{align} 
\end{lemma}
\begin{proof}
Note that in order to prove the left-hand-side of (\ref{iso: todd nc to c}) it suffices to show that the following square is commutative:

\begin{align}
\label{diagram: mult todd alg}
\begin{split}
\xymatrix@C=6.5pc @R=4pc{
\bigoplus_{(T,p,q)}\cZ^*(X \times T)_\QQ/_{\!\sim_\rat} \ar[d]_-{\bigoplus (-\cdot \sqrt{\td_{X \times T}})}^-{\simeq} \ar[r]^-{\bigoplus((id \times p)^*-(id \times q)^*)} & \cZ^*(X)_\QQ/_{\!\sim_\rat}    \ar[d]^-{-\cdot \sqrt{\td_X}}_-{\simeq}\\
\bigoplus_{(T,p,q)}\cZ^*(X \times T)_\QQ/_{\!\sim_\rat}  \ar[r]_-{\bigoplus((id \times p)^*-(id \times q)^*)}  & \cZ^*(X)_\QQ/_{\!\sim_\rat},}
\end{split}
\end{align}

\noindent where $p$ and $q$ stand for $k$-rational points of a smooth connected $k$-scheme $T$. Let $\pi_X$ and $\pi_T$ be, respectively, the projections from $X \times T$ to $X$ and $T$. Consider the pull-backs $\pi_X^*: \cZ^*(X)_\QQ/_{\!\sim_\rat} \rightarrow \cZ^*(X \times T)_\QQ/_{\!\sim_\rat}$ and $\pi_T^*: \cZ^*(T)_\QQ/_{\!\sim_\rat} \rightarrow \cZ^*(X \times T)_\QQ/_{\!\sim_\rat}$.

It is known, by the proof of \cite[Lemma 10.6]{Huyb06}, that $\sqrt{\td_{X \times T}}=\pi_X^*(\sqrt{\td_X})\cdot \pi_T^*(\sqrt{\td_T})$. Consequently, we conclude that
\begin{align*}
    (id \times p)^*(\sqrt{\td_{X \times T}})&=(id \times p)^*(\pi_X^*(\sqrt{\td_X})\cdot \pi_T^*(\sqrt{\td_T}))\\
    &=(id \times p)^*(\pi_X^*(\sqrt{\td_X})) \cdot (id \times p)^*(\pi_T^*(\sqrt{\td_T}))\\
    &= (\pi_X \circ (id \times p))^*(\sqrt{\td_X}) \cdot (\pi_T \circ (id \times p))^*(\sqrt{\td_T})\\
    &= id^*(\sqrt{\td_X}) \cdot (\pi_T \circ (id \times p))^*(\sqrt{\td_T})\\
    &=\sqrt{\td_X} \cdot (\pi_T \circ (id \times p))^*(\sqrt{\td_T}).
\end{align*}

Recall from \cite[\S 5.2]{Huyb06} that the degree zero term of $\td_T$, in $\mathrm{H}^0(T;\QQ)$, is $1$. Therefore, choose for $\sqrt{\td_T}$ the cohomology class whose degree zero term is 1. Since $\cZ^*(\Spec (k))_\QQ/_{\!\sim_\rat}=\QQ$ and $p^*$ is a graded ring morphism, the following equalities hold for every $\nu \in \cZ^*(X \times T)_\QQ/_{\!\sim_\rat}$.
\begin{align*}
    (id \times p)^*(\nu \cdot \sqrt{\td_{X \times T}})&= (id \times p)^*(\nu) \cdot (id \times p)^*(\sqrt{\td_{X \times T}})\\
    &=(id \times p)^*(\nu) \cdot \sqrt{\td_X} \cdot (\pi_T \circ (id \times p))^*(\sqrt{\td_T})\\
    &= (id \times p)^*(\nu) \cdot \sqrt{\td_X} \cdot (p \circ \pi_{\Spec(k)})^*(\sqrt{\td_T}) \\
    &= (id \times p)^*(\nu) \cdot \sqrt{\td_X} \cdot (\pi_{\Spec(k)}^* \circ p^*)(\sqrt{\td_T})\\
    &= (id \times p)^*(\nu) \cdot \sqrt{\td_X} \cdot \pi_{\Spec(k)}^*(1) \\
    &= (id \times p)^*(\nu) \cdot \sqrt{\td_X}.
\end{align*}

\noindent This implies that the above diagram (\ref{diagram: mult todd alg}) is commutative.

Now, in order to prove the right-hand-side of (\ref{iso: todd nc to c}), note that the homomorphism $\{\mu_i\}_{1\leq i \leq m} \longmapsto \mu_1 \times \ldots \times \mu_m$ can be written as the following composition:
\begin{align*}
    \prod_{i=1}^m \cZ^*(X)_\QQ /_{\!\sim_\rat} \xrightarrow{\prod \pi_i^*} \prod_{i=1}^m \cZ^*(X^{\times m})_\QQ/_{\!\sim_\rat} \xrightarrow{\text{mult}} \cZ^*(X^{\times m})_\QQ/_{\!\sim_\rat},
\end{align*}
where $\pi_i$ stands for the $i^{th}$ projection map from $X^{\times m}$ to $X$. Consequently, since the pull-back homomorphism $\pi_i^*$ is multiplicative, if an algebraic cycle $\mu \in \cZ^*(X)_\QQ/_{ \!\sim_\rat}$ is nilpotently trivial, then so is $\mu \cdot \sqrt{\td_{X}}$. Conversely, since the homomorphism $\pi_i^*$ is multiplicative and $\sqrt{\td_{X}}$ is invertible, if $\mu \cdot \sqrt{\td_{X}}$ is nilpotently trivial, then $\mu$ is also necessarily nilpotently trivial.
\end{proof}

\section{Kernels}

In this section, we consider the quotients between the different equivalence relations introduced in \S \ref{equiv rel}. We start by fixing some useful notations that will be used in the remainder of the paper.

\begin{notation}
\renewcommand{\labelenumi}{(\roman{enumi})}
Let $X$ be a smooth proper $k$-scheme and $\cA$ a smooth proper dg category (over $k$).
\begin{enumerate}
    \item Given equivalence relations $\sim_1$ and $\sim_2$ on a set $S$, we will write $\sim_1 \succeq \sim_2$ if, for all $a,$ $b \in S$, $a \sim_1 b$ implies $a \sim_2 b$.
    \item Given an equivalence relation $\sim$ on $\cZ^*(X)_\QQ$, resp. on $K_0(\cA)_\QQ$, let us write 
    \begin{align*}
    \cZ^*_{\sim}(X)_\QQ:=\{\alpha \in \cZ^*(X)_\QQ | \alpha \sim 0 \} \text{ resp. } K_{0,\sim}(\cA)_\QQ:=\{\alpha \in K_0(\cA)_\QQ | \alpha \sim 0 \}
    \end{align*}
    for the $\QQ$-subspace of $\sim$-trivial elements.
    \item Given equivalence relations $\sim_1 \, \succeq \, \sim_2$ on $\cZ^*(X)_\QQ$, resp. on $K_0(\cA)_\QQ$, let us write
    \begin{align*}
    \cZ^*_{\sim_2/\sim_1}(X)_\QQ:=\dfrac{\cZ^*_{\sim_2}(X)_\QQ}{\cZ^*_{\sim_1}(X)_\QQ} \text{ resp. } K_{0,\sim_2/\sim_1}(\cA)_\QQ:=\dfrac{K_{0,\sim_2}(\cA)_\QQ}{K_{0,\sim_1}(\cA)_\QQ}
    \end{align*} for the associated quotient.
    \item Given equivalence relations $\sim_1 \, \succeq \, \sim_2$ on $\cZ^*(X)_\QQ$, resp. on $K_0(\cA)_\QQ$, let us write
    \begin{align*}
    q^{X}_{\sim_2/\sim_1}: \cZ^*(X)_\QQ/_{\!\sim_1} \twoheadrightarrow \cZ^*(X)_\QQ/_{\!\sim_2} \text{ resp. } q^{\cA,\nc}_{\sim_2/\sim_1}: K_0(\cA)_\QQ/_{\!\sim_1} \twoheadrightarrow K_0(\cA)_\QQ/_{\!\sim_2}
    \end{align*}
    for the quotient map.
    \item Given equivalence relations $\sim_1 \, \succeq \, \sim_2$ on $K_0(\perfdg(X))_\QQ$, let $q^{X,\nc}_{\sim_2/\sim_1}: K_0(\perfdg(X))_\QQ/_{\!\sim_1} \twoheadrightarrow K_0(\perfdg(X))_\QQ/_{\!\sim_2}$ be the quotient map.
\end{enumerate}
\end{notation}  

\begin{remark}
\label{remark: trivial leq 2}
\renewcommand{\labelenumi}{(\roman{enumi})}
\begin{enumerate}
    \item The equivalence relations $\sim_\alg$, $\sim_\nil$ and $\sim_\num$ agree for smooth proper $k$-schemes $X$ of $\dim \leq 2$; see \cite[\S 3.2.7]{YAnd04} and \cite[\S 19.3.5]{Fulton2nd_ed}. Moreover, $\sim_\rat$, $\sim_\alg$, $\sim_\nil$ and $\sim_\num$ agree for $0$-dimensional $k$-schemes.
    \item In general, we have $\ker(q^{X}_{\sim_2/\sim_1}) \neq 0$; consult Remark \ref{remark: not always true}.
    \item For every smooth projective complex curve $X$ of positive genus $g$, we have that $\cZ^*_{\sim_\alg/\sim_\rat}(X)_\QQ \simeq (\QQ/\ZZ)^{\oplus 2 g}$. Indeed, as $X$ is a curve, we have $\cZ^*_{\sim_\alg/\sim_\rat}(X)_\QQ = \cZ^1_{\sim_\alg/\sim_\rat}(X)_\QQ$ which, thanks to \cite[19.3.5]{Fulton2nd_ed}, is isomorphic to the $\QQ$-vector space of torsion points of the Albanese variety $\text{Alb}(X)$.
    Since $X$ is a curve, we have $\text{Alb}(X) \simeq \text{Jac}(X)$ (the Jacobian variety of X) and it is  well-known that the $\QQ$-vector space of torsion points of $\text{Jac}(X)$ is isomorphic to $(\QQ/\ZZ)^{\oplus 2 g}$.
    \item Voevodsky's nilpotence conjecture asserts that $\ker(q^{X}_{\sim_\num/\sim_\nil})=0$.
\end{enumerate}
\end{remark}

\begin{remark}
\label{remark: kernel=quocient}
Let $\sim_1, \sim_2 \in \{\sim_\rat, \sim_\alg, \sim_\nil, \sim_\num\}$.
\renewcommand{\labelenumi}{(\roman{enumi})}
\begin{enumerate}
\item Given a smooth proper $k$-scheme $X$, we have $\cZ^*_{\sim_2/\sim_1}(X)_\QQ\simeq\ker(q^X_{\sim_2/\sim_1})$.
\item Given a smooth proper dg category $\cA$ (over $k$), we have $K_{0,\sim_2/\sim_1}(\cA)_\QQ \simeq \ker(q^{\cA,\nc}_{\sim_2/\sim_1})$.
\end{enumerate}
\end{remark}

\begin{corollary}[of Theorem \ref{theo: non commutative to classical rel}]
\label{corollary: nc and c are =}
Given a smooth proper $k$-scheme $X$ and $\sim_1, \sim_2 \in \{\sim_\rat, \sim_\alg, \sim_\nil, \sim_\num\}$, we have an isomorphism $\ker(q^{X}_{\sim_2/\sim_1}) \simeq \ker(q^{X,\nc}_{\sim_2/\sim_1})$. Equivalently, we have an isomorphism $\cZ^*_{\sim_2/\sim_1}(X)_\QQ \simeq K_{0,\sim_2/\sim_1}(\perfdg(X))_\QQ$.
\end{corollary}

\begin{proof}
Recall from Theorem \ref{theo: non commutative to classical rel} that the map (\ref{iso from nc to c}) induces an isomorphism $K_0(\perfdg(X))_\QQ/_{\!\sim} \xrightarrow[]{\simeq} \cZ^*(X)_{\QQ}/_{\!\sim}$, where

\noindent $\sim \in \{\sim_\rat, \sim_\alg, \sim_\nil, \sim_\num\}$.
Consequently, we have the following commutative diagram:

\begin{align}
\label{diagram: kernels}
\begin{split}\xymatrix@C=3pc @R=3pc{
\ker(q^{X,\nc}_{\sim_2/\sim_1}) \ar[r] \ar[d]& K_0(\perfdg(X))_\QQ/_{\!\sim_1} \ar@{->>}[r] \ar[d]^-{(\ref{iso from nc to c})}_-{\simeq}& K_0(\perfdg(X))_\QQ/_{\!\sim_2} \ar[d]^-{(\ref{iso from nc to c})}_-{\simeq}\\
\ker(q^{X}_{\sim_2/\sim_1}) \ar[r] & \cZ^*(X)_{\QQ}/_{\!\sim_1} \ar@{->>}[r] & \cZ^*(X)_{\QQ}/_{\!\sim_2}.}
\end{split}
\end{align}

\noindent Since both rows in (\ref{diagram: kernels}) are exact, we hence conclude that the left-hand-side vertical homomorphism in (\ref{diagram: kernels}) is also invertible.
\end{proof}

\begin{proposition}
\label{proposition: kernel nc factors}
Let $\cA$ be a smooth proper dg category such that $\mathrm{H}^0(\cA)=\langle \cfb, \cfc \rangle$ admits a semi-orthogonal decomposition in the sense of Bondal-Orlov \cite{BondalOrlov02}. Let $\cfb^{\dg}$ and $\cfc^{\dg}$ denote, respectively, the dg enhancement of $\cfb$ and $\cfc$ induced from $\cA$.
Under these assumptions, the inclusions of $\cfb^{\dg}$ and $\cfc^{\dg}$ into $\cA$ induce an isomorphism
\begin{align*}
\ker(q^{\cfb^{\dg},\nc}_{\sim_2/\sim_1}) \oplus \ker(q^{\cfc^{\dg},\nc}_{\sim_2/\sim_1}) \stackrel{\simeq}{\longrightarrow} \ker(q^{\cA,\nc}_{\sim_2/\sim_1}\!) \text{ with } \sim_1,\sim_2 \in \{\sim_\rat, \sim_\alg, \sim_\nil, \sim_\num \}.   
\end{align*}
Equivalently, we have an induced isomorphism $K_{0,\sim_2/\sim_1}(\cfb^{\dg})_\QQ \oplus K_{0,\sim_2/\sim_1}(\cfc^{\dg})_\QQ \xrightarrow{\simeq} K_{0,\sim_2/\sim_1}(\cA)_\QQ$.
\end{proposition}

\begin{proof}
Since the dg category $\cA$ is smooth and proper, it follows from the proof of \cite[Lemma 2.1]{BMT} that the dg categories $\cfb^{\dg}$ and $\cfc^{\dg}$ are also smooth and proper.
Moreover, as explained in \cite[\S 3.2]{Tab18a}, we have the following computations:
\begin{align}
\label{Hom Nvoev Nnum}
\Hom_{\text{NVoev}(k)_\QQ}(U(k)_\QQ, U(\cA)_\QQ) \simeq K_0(\cA)_\QQ/_{\!\sim_\nil} \qquad  \Hom_{\text{NNum}(k)_\QQ}(U(k)_\QQ, U(\cA)_\QQ) \simeq K_0(\cA)_\QQ/_{\!\sim_\num}.
\end{align}

\noindent As $\mathrm{H}^0(\cA)= \langle \cfb, \cfc \rangle$, we deduce from \cite[Proposition 2.2 and Theorem 2.9]{Tab15a} that the induced morphism 
$U(\cfb^{\dg})_\QQ \oplus U(\cfc^{\dg})_\QQ \rightarrow U(\cA)_\QQ$ is invertible.
Therefore, making use of the computations (\ref{Hom Nvoev Nnum}), we obtain induced isomorphisms:
\begin{align*}
K_0(\cfb^{\dg})_\QQ/_{\!\sim_\nil} \oplus  K_0(\cfc^{\dg})_\QQ/_{\!\sim_\nil} \stackrel{\simeq}{\longrightarrow} K_0(\cA)_\QQ/_{\!\sim_\nil} \qquad
K_0(\cfb^{\dg})_\QQ/_{\!\sim_\num} \oplus  K_0(\cfc^{\dg})_\QQ/_{\!\sim_\num} \stackrel{\simeq}{\longrightarrow} K_0(\cA)_\QQ/_{\!\sim_\num}.
\end{align*}
Thanks to the computation (\ref{Hom Nchow}), we have an induced isomorphism $K_0(\cfb^{\dg})_\QQ \oplus  K_0(\cfc^{\dg})_\QQ \xrightarrow[]{\simeq} K_0(\cA)_\QQ$.
Note that thanks to Lemma \ref{lemma: K algebrico 2nd prop additive}, the inclusions of $\cfb^{\dg}$ and $\cfc^{\dg}$ into $\cA$ also induce an isomorphism $K_0(\cfb^{\dg})_\QQ/_{\!\sim_\alg} \oplus  K_0(\cfc^{\dg})_\QQ/_{\!\sim_\alg} \xrightarrow[]{\simeq} K_0(\cA)_\QQ/_{\!\sim_\alg}$. Consequently, we have the following commutative diagram:

\begin{align}
\label{useful diagram decomposition}
\begin{split}
\xymatrix@C=3pc @R=4pc{
\ker(q^{\cfb^{\dg},\nc}_{\sim_2/\sim_1}) \oplus  \ker(q^{\cfc^{\dg},\nc}_{\sim_2/\sim_1}) \ar[r] \ar[d] & K_0(\cfb^{\dg})_\QQ/_{\!\sim_1} \oplus  K_0(\cfc^{\dg})_\QQ/_{\!\sim_1} \ar@{->>}[r] \ar[d]^{\simeq} & K_0(\cfb^{\dg})_\QQ/_{\!\sim_2} \oplus  K_0(\cfc^{\dg})_\QQ/_{\!\sim_2} \ar[d]^{\simeq}\\
\ker(q^{\cA,\nc}_{\sim_2/\sim_1}) \ar[r] & K_0(\cA)_\QQ/_{\!\sim_1} \ar@{->>}[r] & K_0(\cA)_\QQ/_{\!\sim_2}.}
\end{split}    
\end{align}

Since both rows in (\ref{useful diagram decomposition}) are exact, we conclude that the left-hand-side vertical homomorphism in (\ref{useful diagram decomposition}) is also invertible.
\end{proof}

\begin{proposition}
\label{proposition: morita implies equivalent conj}
Let $\cA$ and $\cB$ be two smooth proper dg categories.
\renewcommand{\labelenumi}{(\roman{enumi})}
\begin{enumerate}
    \item Given $F: \cA \rightarrow \cB$ a Morita equivalence, the homomorphism $K_0(F)_\QQ/_{\!\sim}$ is invertible when $\sim \in \{\sim_\rat,\sim_\alg, \sim_\nil, \sim_\num\}$.
    \item For $\sim_1, \sim_2 \in \{\sim_\rat, \sim_\alg, \sim_\nil, \sim_\num\}$ and $\cA$ Morita equivalent to $\cB$, we have $\ker(q^{\cA,\nc}_{\sim_2/\sim_1}) \simeq \ker(q^{\cB,\nc}_{\sim_2/\sim_1})$, or, equivalently, $K_{0,\sim_2/\sim_1}(\cA)_\QQ \simeq K_{0,\sim_2/\sim_1}(\cB)_\QQ$.
\end{enumerate}
\end{proposition}
\begin{proof}
We start by proving item (i). Since $K_0(-)_\QQ$ is an additive invariant, we have that $K_0(F)_\QQ$ is an isomorphism.
Thanks to Lemma \ref{lemma: k algebrico sends morita to iso}, $K_0(F)_\QQ/_{\!\sim_\alg}$ is an isomorphism.
Following \cite[Proposition 2.2 and Theorem 2.9]{Tab15a}, we have that $U(-)$ sends Morita equivalences to isomorphisms. This together with the following computations (see \cite[\S 3.2]{Tab18a})
\begin{align*}
\Hom_{\text{NVoev}(k)_\QQ}(U(k)_\QQ, U(-)_\QQ) \simeq K_0(-)_\QQ/_{\!\sim_\nil} \qquad \Hom_{\text{NNum}(k)_\QQ}(U(k)_\QQ, U(-)_\QQ) \simeq K_0(-)_\QQ/_{\!\sim_\num},
\end{align*}
is enough to conclude that both $K_0(F)_\QQ/_{\!\sim_\nil}$, and $K_0(F)_\QQ/_{\!\sim_\num}$ are isomorphisms.

Now, note that item (ii) follows from item (i). \end{proof}

\begin{proposition}
\label{proposition: azumaya imples equivalent conj}
Let $X$ be a smooth proper $k$-scheme and $\BB_0$ an Azumaya algebra over $X$.
\renewcommand{\labelenumi}{(\roman{enumi})}
\begin{enumerate}
    \item The canonical dg functor $i_{X,\BB_0}:\perfdg(X) \rightarrow \perfdg(X,\BB_0)$ induces an isomorphism 
    \begin{align*}
    K_0(i_{X,\BB_0})_\QQ/_{\!\sim}: K_0(\perfdg(X))_\QQ/_{\!\sim} \longrightarrow K_0(\perfdg(X,\BB_0))_\QQ/_{\!\sim} \text{ with } \sim \in \{\sim_\rat,\sim_\alg, \sim_\nil, \sim_\num\}.
    \end{align*}
    \item For $\sim_1, \sim_2 \in \{\sim_\rat, \sim_\alg, \sim_\nil, \sim_\num\}$, we have $\ker(q^{X,\nc}_{\sim_2/\sim_1}) \simeq \ker(q^{\perfdg(X,\BB_0),\nc}_{\sim_2/\sim_1})$, or, equivalently, an isomorphism
\begin{align*}
K_{0,\sim_2/\sim_1}(\perfdg(X))_\QQ \simeq K_{0,\sim_2/\sim_1}(\perfdg(X,\BB_0))_\QQ.   
\end{align*}    

\end{enumerate}
\end{proposition}
\begin{proof}
To prove item (i), we start by noting that, following \cite[Corollary 3.1]{TabVB15}, we have $K_0(\perfdg(X))_\QQ \simeq K_0(\perfdg(X,\BB_0))_\QQ$.
Moreover, thanks to Lemma \ref{lemma: azumaya algebras}, we have that $ K_0(\perfdg(X))_\QQ/_{\!\sim_\alg} \simeq K_0(\perfdg(X,\BB_0))_\QQ/_{\!\sim_\alg}$.
Following \cite[Propositions 8.3 and 8.7]{TabVB15}, the canonical dg functor $i_{X,\BB_0}$ induces an isomorphism $U(\perfdg(X))_\QQ \simeq U(\perfdg(X,\BB_0))_\QQ$. So, the same reasoning in Proposition \ref{proposition: morita implies equivalent conj} allows us to conclude that
\begin{align*}
K_0(\perfdg(X))_\QQ/_{\!\sim_\nil} \simeq K_0(\perfdg(X,\BB_0))_\QQ/_{\!\sim_\nil} \qquad K_0(\perfdg(X))_\QQ/_{\!\sim_\num} \simeq K_0(\perfdg(X,\BB_0))_\QQ/_{\!\sim_\num}.
\end{align*}
This finishes the proof of item (i). 
To prove item (ii), note that item (i) yields the following commutative diagram:
\begin{align}
\label{useful diagram azumaya}
\begin{split}
\xymatrix@C=3pc @R=4pc{
\ker(q^{X,\nc}_{\sim_2/\sim_1}) \ar[r] \ar[d] & K_0(\perfdg(X))_\QQ/_{\!\sim_1} \ar@{->>}[r] \ar[d]^{\simeq} & K_0(\perfdg(X))_\QQ/_{\!\sim_2} \ar[d]^{\simeq}\\
\ker(q^{\perfdg(X,\BB_0),\nc}_{\sim_2/\sim_1}) \ar[r] & K_0(\perfdg(X,\BB_0))_\QQ/_{\!\sim_1} \ar@{->>}[r] & K_0(\perfdg(X,\BB_0))_\QQ/_{\!\sim_2}.}
\end{split}
\end{align}
\noindent Since both rows in (\ref{useful diagram azumaya}) are exact, the left-hand-side vertical homomorphism in (\ref{useful diagram azumaya}) is also invertible.
\end{proof}

\section{Applications}

In this section, making use of the mathematical tools developed in \S 3-\S 4, we prove Theorem \ref{main theo} (as well as Theorem \ref{Theorem: HPD}). The proofs are, in most cases, adaptations of the ones given in \cite{BMT} and \cite{OP}.

\subsection{Full exceptional collection}
\begin{corollary}[of Proposition \ref{proposition: kernel nc factors}]
\label{corollary: full}
Let $X$ be a smooth proper $k$-scheme such that $\perf(X)$ admits a full exceptional collection, cf. \cite[\S1.1]{kuzn14}. The equivalence relations $\sim_\alg$ and $\sim_\num$ agree on $\cZ^*(X)_\QQ$.
\end{corollary}
\begin{proof}
Just consider Remark \ref{remark: trivial leq 2}(i), Corollary \ref{corollary: nc and c are =} and Proposition \ref{proposition: kernel nc factors}.
\end{proof}

\subsection{Severi-Brauer varieties}
\label{subsection: severi-brauer}

\begin{theorem}
\label{theorem: severi-brauer}
Let $X$ be a Severi-Brauer variety. The equivalence relations $\sim_\alg$ and $\sim_\num$ agree on $\cZ^*(X)_\QQ$.
\end{theorem}
\begin{proof}
Let $A$ be the unique central simple $k$-algebra such that $X=\mathrm{SB}(A)$. Following \cite[Theorem 4.1]{Ber09}, we have the semi-orthogonal decomposition:
\begin{align*}
    \perf(X)= \langle \perf(k), \perf(A), \ldots, \perf(A^{\otimes n})\rangle,
    \end{align*}
where $n$ stands for the dimension of $X$. Since the tensor product of central simple $k$-algebra is still a central simple $k$-algebra, Remark \ref{remark: trivial leq 2}(i), Corollary \ref{corollary: nc and c are =} and Propositions \ref{proposition: kernel nc factors}, \ref{proposition: azumaya imples equivalent conj}(ii) together imply that the equivalence relations $\sim_\alg$ and $\sim_\num$ agree on $\cZ^*(X)_\QQ$.
\end{proof}

\subsection{Quadric fibrations}
\label{subsection: quadric}
Take $S$ a smooth projective $k$-scheme and $q: Q \rightarrow S$ a flat quadric fibration of relative dimension $n$ with $Q$ smooth.
Let $C_0$ be the sheaf of even parts of the Clifford algebra associated to $q$; see \cite[\S 1]{ABB14} \cite[\S 3]{Kuzn08}.
Recall from \cite[\S 3.5-\S 3.6]{Kuzn08} that when the discriminant divisor of $q$ is smooth and $n$ is even (resp., odd) we have a discriminant double cover $\widetilde{S} \rightarrow S$ (resp., a square root stack $\widehat{S}$) equipped with an Azumaya algebra $\BB_0$.

\begin{theorem}
Under the above assumptions, with $\sim_1, \sim_2 \in \{\sim_\rat, \sim_\alg, \sim_\nil, \sim_\num\}$, the following holds:
\renewcommand{\labelenumi}{(\roman{enumi})}
\begin{enumerate}
 \item We have $\cZ^*_{\sim_2/\sim_1}(Q)_\QQ\simeq K_{0,\sim_2/\sim_1}(\perfdg(S, C_0))_\QQ \bigoplus \cZ^*_{\sim_2/\sim_1}(S)^{\oplus n}_\QQ$.
 \item If the discriminant divisor of $q$ is smooth and $n$ is even, then $\cZ^*_{\sim_2/\sim_1}(Q)_\QQ \simeq  \cZ^*_{\sim_2/\sim_1}(\widetilde{S})_\QQ \bigoplus \cZ^*_{\sim_2/\sim_1}(S)^{\oplus n}_\QQ$.
 \item If the discriminant divisor of $q$ is smooth and $n$ is odd, then $\cZ^*_{\sim_2/\sim_1}(Q)_\QQ \simeq K_{0,\sim_2/\sim_1}(\perfdg(\widehat{S}, \BB_0))_\QQ \bigoplus \cZ^*_{\sim_2/\sim_1}(S)^{\oplus n}_\QQ$. 
\end{enumerate}
\label{theorem: quadric fibrations}
\end{theorem}
\begin{proof}
Follow the reasoning of the proof of \cite[Theorem 1.2]{BMT} with the following changes:
to prove item (i), use Corollary \ref{corollary: nc and c are =} instead of \cite[Theorem 1.1]{BMT} and Proposition \ref{proposition: kernel nc factors} instead of \cite[equation (5.2)]{BMT} (in order to conclude that $\ker(q^{Q,\nc}_{\sim_2/\sim_1})$ is isomorphic to $\ker(q^{\perfdg(S, C_0),\nc}_{\sim_2/\sim_1}) \bigoplus \ker(q^{S, \nc}_{\sim_2/\sim_1})^{\oplus n}$);
to prove item (ii), consider Proposition \ref{proposition: azumaya imples equivalent conj} and use Corollary \ref{corollary: nc and c are =} instead of \cite[Theorem 1.1]{BMT};
finally, in order to prove (iii), it is enough to replace the reference \cite[Theorem 1.1]{BMT} by Corollary \ref{corollary: nc and c are =}.
\end{proof}

\begin{corollary}
\label{corollary: certain quadrics}
When $\dim(S)\leq 2$, the following holds:
\renewcommand{\labelenumi}{(\roman{enumi})}
\begin{enumerate}
\item If the discriminant divisor of $q$ is smooth and $n$ is even, then the equivalence relations $\sim_\alg$ and $\sim_\num$ on $\cZ^*(Q)_\QQ$ agree.
 \item If the discriminant divisor of $q$ is smooth and $n$ is odd, then $\cZ^*_{\sim_\num/\sim_\alg}(Q)_\QQ \simeq K_{0,\sim_\num/\sim_\alg}(\perfdg(\widehat{S}, \BB_0))_\QQ$. 
\end{enumerate}
\end{corollary}
\begin{proof}
Since $\dim(\widetilde{S})=\dim (S)$, the proof follows from Remark \ref{remark: trivial leq 2}(i). 
\end{proof}

\subsection{Intersection of quadrics}

Let $X$ be a smooth complete intersection of $r$ quadric hypersurfaces in $\PP^m$. The linear span of these $r$ quadrics gives rise to a hypersurface
$Q \subseteq \PP^{r-1} \times \PP^m$,
and the projection into the first factor to a flat quadric fibration $q:Q \rightarrow \PP^{r-1}$ of relative dimension $m-1$.
\begin{theorem}
\label{theorem: intersection of quadrics}
Under the above assumptions, with $\sim_1, \sim_2 \in \{\sim_\rat, \sim_\alg, \sim_\nil, \sim_\num\}$, the following holds:
\renewcommand{\labelenumi}{(\roman{enumi})}
\begin{enumerate}
 \item We have $\cZ^*_{\sim_2/\sim_1}(X)_\QQ \simeq K_{0,\sim_2/\sim_1}(\perfdg(\PP^{r-1}, C_0))_\QQ$.
 \item If the discriminant divisor of $q$ is smooth and $m$ is odd, then $\cZ^*_{\sim_2/\sim_1}(X)_\QQ \simeq \cZ^*_{\sim_2/\sim_1}(\widetilde{\PP^{r-1}})_\QQ$.
 \item If the discriminant divisor of $q$ is smooth and $m$ is even, then $\cZ^*_{\sim_2/\sim_1}(X)_\QQ \simeq K_{0,\sim_2/\sim_1}(\perfdg(\widehat{\PP^{r-1}}, \BB_0))_\QQ$.
\end{enumerate}
\end{theorem}

\begin{proof}
The proof is similar to the one given in \cite[\S 7]{BMT} with the following adaptations: to prove item (i), replace the reference to the proof of \cite[Theorem 1.2(i)]{BMT} by the actual proof of Theorem \ref{theorem: quadric fibrations}(i) and consider Corollary \ref{corollary: nc and c are =} instead of \cite[Theorem 1.1]{BMT};
the proofs of items (ii)-(iii) follow a similar reasoning to the proofs of Theorem \ref{theorem: quadric fibrations}(ii)-(iii).
\end{proof}

\begin{corollary}
\label{corollary: certain intersections}
When $r\leq 3$, the following holds:
\renewcommand{\labelenumi}{(\roman{enumi})}
\begin{enumerate}
 \item If the discriminant divisor of $q$ is smooth and $m$ is odd, then the equivalence relations $\sim_\alg$ and $\sim_\num$ on $\cZ^*(X)_\QQ$ agree.
 \item If the discriminant divisor of $q$ is smooth, $m$ is even and $k$ is algebraically closed, then the equivalence relations $\sim_\alg$ and $\sim_\num$ on $\cZ^*(X)_\QQ$ agree.
\end{enumerate}
\end{corollary}

\subsection{Moishezon manifolds}

A \textit{Moishezon manifold} $X$ is a compact complex manifold such that the field of meromorphic functions on each component of $X$ has transcendence degree equal to the dimension of the component.
As proved in \cite{Mois66}, $X$ is a smooth projective $\CC$-scheme if and only if it admits a K\"{a}hler metric. In the remaining cases, it is shown in \cite{MArtin70} that $X$ is a proper algebraic space over $\CC$.
Let $Y \rightarrow \PP^2$ be one of the non-rational conic bundles described by Artin and Mumford in \cite{AM72}, and $X \rightarrow Y$ a small resolution. In this case, $X$ is a smooth (not
necessarily projective) Moishezon manifold.

\begin{theorem}
\label{theorem: certain moishezon}
The equivalence relations $\sim_\alg$ and $\sim_\num$ on $K_0(\perfdg(X))_\QQ$ agree.
\end{theorem}
\begin{proof}
The proof of \cite[Theorem 1.14]{BMT} can be adapted as follows:
replace the reference to the proof of \cite[Thereom 1.2(i)]{BMT} by the proof of Theorem \ref{theorem: quadric fibrations}(i), consider Remark \ref{remark: trivial leq 2}(i), and use Proposition \ref{proposition: kernel nc factors} instead of \cite[Theorem 1.2]{BMT}.
\end{proof}

\subsection{Cubic fourfolds and Gushel-Mukai fourfolds}

Recall that a \textit{cubic fourfold} is a smooth complex hypersurface of degree 3 in $\PP^5$ and consult \cite[\S 2.2]{OP} for the definition of an (ordinary/ special) Gushel-Mukai fourfold.

In what follows, we adapt \cite[Theorems (A)-(D)]{OP} to our context.

\begin{theorem}
\label{theorem: certain cubic fourfolds}
The equivalence relations $\sim_\alg$ and $\sim_\num$ on $\cZ^*(X)_\QQ$ agree when $X$ is a cubic fourfold or an ordinary generic Gushel-Mukai fourfold.
\label{theorem (A)}
\end{theorem}
\begin{proof}
Apply the same reasoning of the proof of \cite[Theorem (A)]{OP}, use Remark \ref{remark: trivial leq 2}(i), and replace the equivalence relation $\sim_\nil$ by the equivalence relation $\sim_\alg$.
\end{proof}

\begin{remark}
\renewcommand{\labelenumi}{(\roman{enumi})}
\begin{enumerate}
\item Let $X$ be a cubic fourfold. Recall from \cite{kuzn16} that the Kuznetsov category $\cA_X$ of $X$ is defined as a certain semi-orthogonal component of $\perf(X)$. Therefore, thanks to Theorem \ref{theorem (A)}, Corollary \ref{corollary: nc and c are =}, and Proposition \ref{proposition: kernel nc factors}, we have $\ker(q^{\cA_X^{\dg},\nc}_{\sim_\num/\sim_\alg})= 0$, i.e., the equivalence relations $\sim_\alg$ and $\sim_\num$ on $\cA_X^{\dg}$ agree, where $\cA_X^{\dg}$ denotes the dg enhancement of $\cA_X$ induced from $\perfdg(X)$.
\item Let $X$ be a Gushel-Mukai $n$-fold. Recall from \cite{KuznPer18} that the Gushel-Mukai category $\cA_X$ of $X$ is defined as a certain semi-orthogonal component of $\perf(X)$. Therefore, for $X$ a Gushel-Mukai fourfold, Theorem \ref{theorem (A)}, Corollary \ref{corollary: nc and c are =}, and Proposition \ref{proposition: kernel nc factors}, imply that $\ker(q^{\cA_X^{\dg},\nc}_{\sim_\num/\sim_\alg})=0$, i.e., the equivalence relations $\sim_\alg$ and $\sim_\num$ on $\cA_X^{\dg}$ agree.
\end{enumerate}
\end{remark}

Under the above notations, we obtain the analogous of \cite[Theorem (B)]{OP}:

\begin{theorem}
The equivalence relations $\sim_\alg$ and $\sim_\num$ on $K_0(\cA_X^{\dg})_\QQ$ agree when $X$ is a cubic fourfold or an ordinary generic Gushel-Mukai fourfold.
\label{theorem (B)}
\end{theorem}

We now adapt \cite[Theorems (C) and (D)]{OP} to our context:

\begin{theorem}
The equivalence relations $\sim_\alg$ and $\sim_\num$ on $\cZ^*(X)_\QQ$ agree when $X$ is a generic Gushel-Mukai fourfold containing a plane $P$ of type $\Gr(2,3)$.
\label{theorem (C)}
\end{theorem}

\begin{proof}
The proof is similar to the one given in \cite[\S 5]{OP}. Just bear in mind that one needs to use Corollary \ref{corollary: nc and c are =}, Propositions \ref{proposition: kernel nc factors} and \ref{proposition: morita implies equivalent conj}, and also to consider Theorem \ref{theorem (B)} instead of \cite[Theorem (B)]{OP}.
\end{proof}

\begin{theorem}
\label{theorem (D)}
The equivalence relations $\sim_\alg$ and $\sim_\num$ on $\cZ^*(X)_\QQ$ agree when $X$ is an ordinary Gushel-Mukai fourfold containing a quintic del Pezzo surface.
\end{theorem}
\begin{proof}
Note first that the semi-orthogonal decomposition of $\perf (X)$ consists only of $\cA_X$ and of exceptional objects, \cite[Proposition 2.3]{KuznPer18}. Therefore, Corollary \ref{corollary: nc and c are =} and a repeated use of Proposition \ref{proposition: kernel nc factors} imply that $\ker(q^X_{\sim_\num/\sim_\alg})$ and $\ker(q^{\cA_X^{\dg},\nc}_{\sim_\num/\sim_\alg})$ are isomorphic. 
To finish the proof, just consider Corollary \ref{corollary: nc and c are =}, Proposition \ref{proposition: morita implies equivalent conj}, and follow the proof of \cite[Theorem (D)]{OP}.
\end{proof}
We are in conditions to generalize Theorem \ref{theorem (B)}.
\begin{theorem}
The equivalence relations $\sim_\alg$ and $\sim_\num$ on $K_0(\cA_X^{\dg})_\QQ$ agree when $X$ is a cubic fourfold, an ordinary generic Gushel-Mukai fourfold, an ordinary Gushel-Mukai fourfold containing a plane of type $\Gr(2,3)$ or an ordinary Gushel-Mukai fourfold containing a quintic del Pezzo surface.
\label{generalized theorem (B)}
\end{theorem}

\subsection{K{\"u}chle fourfolds}

A K{\"u}chle fourfold is the zero locus of a global section of a certain vector bundle on a specific Grassmannian. In particular, a \textit{K{\"u}chle fourfold of type} $c_7$, denoted by $X_{c_7}$, is the zero locus of a global section of the vector bundle $\Lambda^2\cU^\perp(1) \oplus \cO(1)$ on $\Gr(3,8)$, where $\cU^\perp$
is the tautological vector subbundle of rank $5$ on the Grassmannian $\Gr(3, 8)$ and $\cO(1)$ stands for the ample generator of its Picard group; consult \cite{Kuchle95, Kuzn15} for further details.

\begin{theorem}
\label{theorem: certain kuchel fourfolds}
The equivalence relations $\sim_\alg$ and $\sim_\num$ on $\cZ^*(X_{c_7})_\QQ$ agree.
\end{theorem}
\begin{proof}
Thanks to \cite[Corollary 4.12]{Kuzn15}, the category $\perf(X_{c_7})$ admits a semi-orthogonal decomposition with 6 exceptional line bundles and a noncommutative K$3$ category $\cA_X$ which is (Fourier-Mukai) equivalent to the non-trivial part of the derived category of a cubic fourfold $Z$.
Consequently, Proposition \ref{proposition: kernel nc factors}, Remark \ref{remark: trivial leq 2}(i), and Theorem \ref{theorem (B)}, imply that $\ker(q^{X_{c_7},\nc}_{\sim_\num/\sim_\alg}) \simeq \ker(q^{\cA_X^{\dg},\nc}_{\sim_\num/\sim_\alg})$ is trivial. Thanks to Corollary \ref{corollary: nc and c are =}, this implies that $\cZ^*_{\sim_\num/\sim_\alg}(X_{c_7})_\QQ$ is trivial.
\end{proof}

\begin{remark}
A consequence of Theorem \ref{theorem: certain kuchel fourfolds} is that Voevodsky's nilpotence conjecture holds for $X_{c_7}$.
To the best of the author's knowledge, this proves Voevodsky's nilpotence conjecture in new cases.
Note that $X_{c_7}$ is not a Gushel-Mukai fourfold because it has Picard number greater than 1; see \cite[\S 1]{DebarKuzn18}.
\end{remark}

\subsection{Family of sextic del Pezzo surfaces}

A \textit{sextic du Val del Pezzo surface} is a normal integral projective surface $X$ with at worst du Val singularities and ample anticanonical class such that $K^2_X=6$. 
Take $S$ and $T$ smooth projective $k$-schemes and $f:T \rightarrow S$ a \textit{du Val family of sextic del Pezzo surfaces}, i.e., $f$ is a flat morphism such that for every geometric point $s \in S$ the fiber $T_s$ of $T$ over $S$ is a sextic du Val del Pezzo surface.
Following \cite[\S 5]{kuzn18}, with $d=2,3$, let $\mathcal{M}_d$ denote the relative moduli stack of semi-stable sheaves on fibers of $T$ over $S$ with Hilbert polynomial $h_d(t):=(3t+d)(t+1)$ and $Z_d$ the coarse moduli space of $\mathcal{M}_d$. Consequently, there are finite flat morphisms $Z_2 \rightarrow S$ and $Z_3 \rightarrow S$ with degree $2$ and $3$, respectively. 

\begin{theorem}
\label{theorem: certain sextic del Pezzo}
Let $f:T \rightarrow S$ be a du Val family of sextic del Pezzo surfaces, and assume that the characteristic of $k$ is not $2$ neither $3$. Under these conditions, with $\sim_1, \sim_2 \in \{\sim_\rat, \sim_\alg, \sim_\nil, \sim_\num\}$, we have an isomorphism:
\begin{align*}
 \cZ^*_{\sim_2/\sim_1}(T)_\QQ \simeq \cZ^*_{\sim_2/\sim_1}(S)_\QQ \oplus \cZ^*_{\sim_2/\sim_1}(Z_2)_\QQ \oplus \cZ^*_{\sim_2/\sim_1}(Z_3)_\QQ.
\end{align*}
\end{theorem}
\begin{proof}
Following \cite[Theorem 5.2 and Proposition 5.11]{kuzn18}, the category $\perf (T)$ admits a semi-orthogonal decomposition
\begin{align*}
    \perf(T)=\langle \perf(S), \perf(Z_2,\BB_2), \perf(Z_3,\BB_3)\rangle,
\end{align*}
where $\cF_2$ and $\cF_3$ are, resp., certains sheafs of Azumaya algebras over $Z_2$, and $Z_3$, of order 2, and 3, resp.. By considering Propositions \ref{proposition: kernel nc factors} and \ref{proposition: azumaya imples equivalent conj}, we hence conclude that
\begin{align*}
 K_{0,\sim_2/\sim_1}(\perfdg(T))_\QQ \simeq K_{0,\sim_2/\sim_1}(\perfdg(S))_\QQ \oplus K_{0,\sim_2/\sim_1}(\perfdg(Z_2))_\QQ \oplus K_{0,\sim_2/\sim_1}(\perfdg(Z_3))_\QQ.
\end{align*}
\noindent Now, an application of Corollary \ref{corollary: nc and c are =} finishes the proof.
\end{proof}

\begin{corollary}
\label{corollary: sextic del Pezzo}
Let $f:T \rightarrow S$ be a du Val family of sextic del Pezzo surface and assume that the characteristic of $k$ is not $2$ neither $3$. If $\dim (S) \leq 2$, then the equivalence relations $\sim_\num$ and $\sim_\alg$ on $\cZ^*(T)_\QQ$ agree.
\end{corollary}

\begin{remark}
Note that, in the conditions of Corollary \ref{corollary: sextic del Pezzo}, Voevodsky's conjecture holds for $T$.
\end{remark}

\subsection{Homological Projective Duality}
\label{subsection: HPD}
Let $X$ be a smooth projective $k$-scheme equipped with a line bundle $\cL_X(1)$ and let us write $X\rightarrow \PP(V)$ for the associated morphism where $V:=\mathrm{H}^0(X, \cL_X(1))^*$.
Assume that the triangulated category $\perf (X)$ admits a Lefschetz decomposition $\langle \AAA_0, \AAA_1(1), \ldots, \AAA_{i-1}(i-1) \rangle$ with respect to $\cL_X(1)$, where $\AAA_r(r):=\AAA_r \otimes \cL_X(r)$, see \cite[Definition 4.1]{Kuzn07}. Note that $\AAA_r(r) \simeq \AAA_r$. 
Bearing in mind \cite[Definition 6.1]{Kuzn07}, let $Y$ be the Homological Projective (HP)-dual of $X$, $\cL_Y(1)$ the HP-dual line bundle, and $Y\rightarrow \PP(V^*)$ the morphism associated to $\cL_Y(1)$. Given a linear subspace $L\subseteq V^* $, we consider the linear sections $X_L:= X \times_{\PP(V)}\PP(L^\perp)$ and $Y_L:= Y \times_{\PP(V^*)}\PP(L)$. For a survey on HP Duality we invite the reader to consult \cite{kuzn14}.

\begin{theorem}[HPD invariance]
Let $X$ and $Y$ be as above and assume that $X_L$ and $Y_L$ are smooth and that $\dim (X_L)=\dim (X)-\dim(L)$ and $\dim (Y_L)=\dim (Y)-\dim(L^\perp)$. Consider the equivalence relations $\sim_1, \sim_2 \in \{\sim_\rat, \sim_\alg, \sim_\nil, \sim_\num\}$ and assume moreover that $\ker(q^{\AAA_0^{\dg},\nc}_{\sim_2/\sim_1})=0$ (or, equivalently, that $K_{0,\sim_2/\sim_1}(\AAA_0^{\dg})_\QQ=0$), where $\AAA_0^{\dg}$ stands for the dg enhancement of $\AAA_0$ induced from $\perfdg(X)$.
Under these assumptions, we have an isomorphism $\cZ^*_{\sim_2/\sim_1}(X_L)_\QQ \simeq \cZ^*_{\sim_2/\sim_1}(Y_L)_\QQ$.
\label{Theorem: HPD}
\end{theorem}

\begin{remark}
Given a \textit{generic} subspace $L\subseteq V^*$, the sections $X_L$ and $Y_L$ are smooth, and we have $\dim(X_L)=\dim(X)-\dim(L)$ and $\dim(Y_L)=\dim(Y)-\dim(L^{\perp})$. Moreover, by inductive use of Proposition \ref{proposition: kernel nc factors}, we have $\ker(q^{\AAA^{\dg},\nc}_{\sim_2/\sim_1})=0$ whenever $\AAA$ admits a full exceptional collection; see \cite[\S 1.1]{kuzn14} and Remark \ref{remark: trivial leq 2}(i).
This shows that the assumptions of Theorem \ref{Theorem: HPD} are quite mild.
\label{Remark: HPD}
\end{remark}
 
 \begin{proof}
The proof is very similar to the proof given in \cite[\S 9]{BMT}. Follow the same reasoning with the next three differences:
firstly, where they write that a certain conjecture holds we write that the respective quotient $K_{0,\sim_2/\sim_1}(-)_\QQ$ is trivial;
secondly, apply Proposition \ref{proposition: kernel nc factors} to conclude that $K_{0,\sim_2/\sim_1}(\AAA_j^{\dg})_\QQ$, $K_{0,\sim_2/\sim_1}(\cfa_j^{\dg})_\QQ$, and $K_{0,\sim_2/\sim_1}(\BB_j^{\dg})_\QQ$ are trivial for every $j$
(the $\cfa_j$'s are the orthogonal complements of $\AAA_{j+1}$ in $\AAA_j$ and the $\BB_j$'s are the components of the semi-orthogonal decomposition of $\perf(Y)$ obtained by \cite[Theorem 6.3]{Kuzn07});
finally, apply Corollary \ref{corollary: nc and c are =} instead of \cite[Theorem 1.1]{BMT}.
 \end{proof}

\begin{example}[Linear sections of Grassmannians]
Let us apply Theorem \ref{Theorem: HPD} to the case of linear sections of Grassmannians:
\renewcommand{\labelenumi}{(\roman{enumi})}
\begin{enumerate}
\item For $W=k^{\oplus 6}$, let $X_L$ be a generic linear section of codimension $r$ of the Grassmannian $\Gr(2,W)$ under the Pl\"{u}cker embedding, and $Y_L$ the corresponding dual linear section of the cubic Pfaffian $\Pf(4,W^*)$ in $\PP(\Lambda^2 W^*)$. Note that $X_L$ and $Y_L$ are smooth and that $\dim(X_L)=8-r$ and $\dim(Y_L)=r-2$ when $r \leq 6$.
\item For $W=k^{\oplus 7}$, let $X_L$ be a generic linear section of codimension $r$ of the Grassmannian $\Gr(2,W)$ under the Pl\"{u}cker embedding, and $Y_L$ the corresponding dual linear section of the cubic Pfaffian $\Pf(4,W^*)$ in $\PP(\Lambda^2 W^*)$. Note that $X_L$ and $Y_L$ are smooth and that $\dim(X_L)=10-r$ and $\dim(Y_L)=r-4$ when $r \leq 10$.
\end{enumerate}

Note also that \cite[(11) and (12)]{Kuzn06}, Proposition \ref{proposition: kernel nc factors}, and Remark \ref{remark: trivial leq 2}, imply that for both classes (i)-(ii) there is a Lefschetz decomposition of $\perf(\Gr(2,W))$ and that 
$\ker(q^{\AAA_0^{\dg},\nc}_{\sim_2/\sim_1})=0$, where $\AAA_0$ is the first component of the Lefschetz decomposition of $\perf(\Gr(2,W))$.
\begin{corollary}
\label{corollary: grassmannians}
Let $X_L$ and $Y_L$ be as in the above classes (i)-(ii) and $\sim_1, \sim_2 \in \{\sim_\rat, \sim_\alg, \sim_\nil, \sim_\num\}$. Under the assumption that $X_L$ and $Y_L$ are smooth, we have $\cZ^*_{\sim_2/\sim_1}(X_L)_\QQ \simeq \cZ^*_{\sim_2/\sim_1}(Y_L)_\QQ$. Moreover, the equivalence relations $\sim_\num$ and $\sim_\alg$ on $\cZ^*(X_L)$ agree when $r \leq 6$ (class (i)), and when $r\leq 6$ and $8 \leq r \leq 10$ (class (ii)).
\end{corollary}
\begin{proof}
The first statement is immediate from Theorem \ref{Theorem: HPD}. The second statement follows from Remark \ref{remark: trivial leq 2}(i), except the case where $X_L$ and $Y_L$ are as in class (i) and $r=5$. In this latter case, just follow the proof in \cite[\S 8]{BMT} and consider Remark \ref{remark: trivial leq 2}(i).
\end{proof}

\begin{corollary}
\label{corollary: Grassmannian curve genus 1}
For $k=\CC$, let $X_L$ and $Y_L$ be as above in class (i) and let $\dim L=r=3$. Then $\cZ^*_{\sim_\alg/\sim_\rat}(X_L)_\QQ \simeq (\QQ/\ZZ)^{\oplus 2}$.
\end{corollary}
\begin{proof}
In this case, $X_L$ is a 5-fold and $Y_L$ is an elliptic curve; see \cite[\S 10]{Kuzn06}. Therefore, since $Y_L$ is of genus 1, $\cZ^*_{\sim_\alg/\sim_\rat}(Y_L)_\QQ$ is isomorphic to $(\QQ/\ZZ)^{\oplus 2 \times 1}$; see Remark \ref{remark: trivial leq 2}(iii). Consequently, Theorem \ref{Theorem: HPD} implies $\cZ^*_{\sim_\alg/\sim_\rat}(X_L)_\QQ$ isomorphic to $(\QQ/\ZZ)^{\oplus 2}$.
\end{proof}

\begin{corollary}
\label{corollary: Grassmannian curve genus 43}
For $k=\CC$, let $X_L$ and $Y_L$ be as in the above class (ii) and let $\dim L=r$. 
\renewcommand{\labelenumi}{(\roman{enumi})}
\begin{enumerate}
    \item If $r=5$, then $\cZ^*_{\sim_\alg/\sim_\rat}(X_L)_\QQ \simeq (\QQ/\ZZ)^{\oplus 86}$.
    \item If $r=9$, then $\cZ^*_{\sim_\alg/\sim_\rat}(Y_L)_\QQ \simeq (\QQ/\ZZ)^{\oplus 30}$.
\end{enumerate}
\end{corollary}
\begin{proof}
When $r=5$, $X_L$ is a Fano 5-fold of index 2 and $Y_L$ is a curve of genus 43; see \cite[\S 11]{Kuzn06}.
Therefore, we deduce that $\cZ^*_{\sim_\alg/\sim_\rat}(Y_L)_\QQ$ is isomorphic to $(\QQ/\ZZ)^{\oplus 2 \times 43}$; see Remark \ref{remark: trivial leq 2}(iii). Consequently, Theorem \ref{Theorem: HPD} implies that $\cZ^*_{\sim_\alg/\sim_\rat}(X_L)_\QQ$ and $(\QQ/\ZZ)^{\oplus 86}$ are isomorphic.

When $r=9$, $X_L$ is a curve of genus 15 and $Y_L$ is a Fano 5-fold of index 2; see \cite[\S 11]{Kuzn06}. Hence, by combining Theorem \ref{Theorem: HPD} with Remark \ref{remark: trivial leq 2}(iii), we conclude that $\cZ^*_{\sim_\alg/\sim_\rat}(Y_L)_\QQ$, $\cZ^*_{\sim_\alg/\sim_\rat}(X_L)_\QQ$ and $(\QQ/\ZZ)^{\oplus 2 \times 15}$ are all isomorphic to each other.
\end{proof}

\end{example}

\begin{example}[Linear sections of determinantal varieties]\label{example: determinantal}
Let $U$ and $V$ be two $k$-vector spaces of dimensions $m$ and $n$, respectively, with $m \leq n$ and $r$ an integer such that $0<r<m$.
As in \cite{BBF06}, consider the determinantal variety $Z^r_{m,n} \subseteq \PP(U \otimes V)$ defined as the locus of those matrices $M:V \rightarrow U^*$ with rank $\leq r$.
It is known that $Z^r_{m,n}$ admits a canonical Springer resolution of singularities $\cX^r_{m,n}:=\PP(\cQ \otimes U) \rightarrow \Gr(r,U)$, where $\cQ$ stands for the tautological quotient on $\Gr(r,U)$.
Under these notations, let $X_L$ be a generic linear section of codimension $c$ of $\cX^r_{m,n}$ under the map $\cX^r_{m,n}\rightarrow \PP(U \otimes V)$, and $Y_L$ the corresponding dual linear section of $\cX^{m-r}_{m,n}$ under the map $\cX^{m-r}_{m,n} \rightarrow \PP(U^* \otimes V^*)$. From \cite[\S 3]{BBF06} we have that $X_L$ and $Y_L$ are both smooth and, from \cite[\S 3]{Tab20}, we have $\dim (X_L)=r(m+n-r)-1-\dim(L)$ and $\dim (Y_L)=r(m-n-r)-1+\dim(L)$.
Moreover, by \cite[\S 3]{BBF06}, Proposition \ref{proposition: kernel nc factors}, and Remark \ref{remark: trivial leq 2}(i), there is a Lefschetz decomposition of $\perf(\cX^r_{m,n})$ and $\ker(q^{\AAA_0^{\dg},\nc}_{\sim_2/\sim_1})=0$.
\begin{corollary}
\label{corollary: determinantal}
Let $X_L$ and $Y_L$ be as in Example \ref{example: determinantal} and $\sim_1, \sim_2 \in \{\sim_\rat, \sim_\alg, \sim_\nil, \sim_\num\}$. Under these assumptions, we have $\cZ^*_{\sim_2/\sim_1}(X_L)_\QQ \simeq  \cZ^*_{\sim_2/\sim_1}(Y_L)_\QQ$. Moreover, the equivalence relations $\sim_\num and \sim_\alg$ on $\cZ^*(X_L)$ (and on $\cZ^*(Y_L)$) agree whenever $\dim (L)\geq r(m+n-r)-3$ or $\dim (L) \leq 3+r(n-m+r)$.
\end{corollary}
\end{example}

\subsection{Prime Fano threefolds and del Pezzo threefolds}

A \textit{Fano variety} is a smooth proper connected algebraic variety whose anticanonical class is ample. Following \cite[\S 5.4]{kuzn21}, we have that a \textit{prime Fano threefold} is a Fano threefold $X$ with Pic$(X)=\ZZ K_X$ whose genus g$(X)$ is defined from $(-K_X)^3=2\text{g}(X)-2$ and it is known that $1 \leq  \text{g}(X)\leq 12$ and g$(X)\neq 11$. Moreover, for prime Fano threefolds of even genus there is a semi-orthogonal decomposition of $\perf(X)$ with a nontrivial component $\cA_X$. 

In the same way, following \cite[\S 5.4]{kuzn21}, a \textit{del Pezzo threefold} is a Fano threefold $Y$ with $-K_Y=2H$, for a primitive Cartier divisor class $H$ and its degree is defined as d$(Y)=H^3$. It is known that $1\leq \text{d}(Y)\leq 5$ for del Pezzo threefolds of Picard rank 1. In addition, we have that a del Pezzo threefold admits a semi-orthogonal decomposition with a non-trivial component $\cB_Y$; consult \cite[\S3]{kuzn09} for further details.

Let $X$ be a prime Fano threefold with g$(X)\in \{8,10,12\}$. Following \cite[Theorem 3.8]{kuzn09}, there exists a unique del Pezzo threefold $Y$ with degree d$(Y)=\dfrac{\text{g}(X)}{2}-1\in \{3,4,5\}$ such that $\cA_X \simeq \cB_Y$.

\begin{proposition}
\label{proposition: Prime Fano and Del Pezzo}
Let $X$ and $Y$ be as above. Then, we have an isomorphism $\cZ^*_{\sim_\alg/\sim_\rat}(X)_\QQ \simeq \cZ^*_{\sim_\alg/\sim_\rat}(Y)_\QQ$. 
\end{proposition}
\begin{proof}
Proposition \ref{proposition: kernel nc factors} and Remark \ref{remark: trivial leq 2}(i) imply that $K_{0,\sim_\alg/\sim_\rat}(X)_\QQ$ is isomorphic to $K_{0,\sim_\alg/\sim_\rat}(\cA_X^{\dg})_\QQ$.
Similarly, $K_{0,\sim_\alg/\sim_\rat}(Y)_\QQ$ is isomorphic to $K_{0,\sim_\alg/\sim_\rat}(\cB_Y^{\dg})_\QQ$.
Since the categories $\cA_X$ and $\cB_Y$ are (Fourier-Mukai) equivalent, we have, moreover, an isomorphism between $K_{0,\sim_\alg/\sim_\rat}(\cA_X^{\dg})_\QQ$ and $K_{0,\sim_\alg/\sim_\rat}(\cB_Y^{\dg})_\QQ$. This implies that $K_{0,\sim_\alg/\sim_\rat}(X)_\QQ$ and $K_{0,\sim_\alg/\sim_\rat}(Y)_\QQ$ are isomorphic. Consequently, Corollary \ref{corollary: nc and c are =} allows us to conclude that $\cZ^*_{\sim_\alg/\sim_\rat}(X)_\QQ$ is isomorphic to $\cZ^*_{\sim_\alg/\sim_\rat}(Y)_\QQ$.
\end{proof}

\subsection{Fano fourfolds of K3 type}
\label{subsection: fano K3 type}
In \cite{BFMT21} a list of 64 Fano fourfolds of K3 type is given.
By combining Remark \ref{remark: trivial leq 2}(i), Corollary \ref{corollary: nc and c are =}, Propositions \ref{proposition: kernel nc factors}, \ref{proposition: azumaya imples equivalent conj} and Theorem \ref{generalized theorem (B)} one obtains the following theorem, where we use the notation in \cite{BFMT21}.
\begin{theorem}
\label{theorem: fano fourfolds of K3 type}
The equivalence relations $\sim_\alg$ and $\sim_\num$ agree when $X$ belongs to one of the following 59 families Fano fourfolds of type K3:
\renewcommand{\labelenumi}{(\roman{enumi})}
\begin{enumerate}
    \item 16 families of Fano fourfolds of K3 type obtained inside products of flag manifolds, where at least one projection is the blow up of a cubic fourfold: from C-2 to C-17.
    \item 4 families of Fano fourfolds of K3 type obtained inside products of flag manifolds, where at least one projection is a blow up of a Gushel–Mukai fourfold: from GM-18 to  GM-20 and GM-22.
    \item 36 families of Fano fourfolds of K3 type that we obtained inside products of flag manifolds, where at least one projection is a blow up with center birational to a K3 surface, and no cubic or Gushel–Mukai fourfold is involved: from K3-24 to K3-35 and from K3-37 to K3-60.
    \item 3 other families of Fano fourfolds of K3 type: from R-61 to R-63.
\end{enumerate}
\end{theorem}

\begin{remark}
\label{remark: fano not k3 type}
Theorem \ref{theorem: fano fourfolds of K3 type} is still valid for three more families of Fano fourfolds that are not of K3 type. Namely: A-65, A-67 and A-68; consult \cite[Appendix A]{BFMT21}.
\end{remark}

\begin{remark}
Note that Theorem \ref{theorem: fano fourfolds of K3 type} and Remark \ref{remark: fano not k3 type} imply that Voevodsky's nilpotence conjecture holds for all those families of Fano fourfolds.
\end{remark}

\noindent \textbf{Acknowledgments}. I thank Gonçalo Tabuada for discussions regarding \cite{BMT} and Enrico Fatighenti for comments on a previous version of this article which led to the inclusion of \S\ref{subsection: fano K3 type}. This work is funded by national funds through the FCT - Fundação para a Ciência e a Tecnologia, I.P., under the scope of the projects UIDB/00297/2020 and UIDP/00297/2020 (Center for Mathematics and Applications) and the PhD scholarship SFRH/BD/144305/2019.

\bibliographystyle{siam}
\bibliography{biblio}

\end{document}